\documentclass[11pt]{amsart}
\usepackage{amsfonts,amsthm,amsmath,amssymb,mathrsfs}
\usepackage{pgf}                      
\usepackage{hyperref}
\usepackage{tikz}
\usetikzlibrary{arrows}
\usepackage{enumitem}
\usepackage{pdfcomment}    

\newcommand{\standout}[1]{\textbf{#1}}   

\newcommand{\ds}{\displaystyle}

\newcommand{\CC}{\mathbb{C}}

\newcommand{\actby}[1]{{}^{#1}}              
\newcommand{\actbylessspace}[1]{{}^{#1}\!}

\newcommand{\vol}[1]{\text{vol}_{#1}^{\perp}}
\newcommand{\mol}{\bullet}
\newcommand{\HHSSG}[1]{\HH^{#1}(S(V),S(V)\#G)}

\newcommand{\kappaLxy}[5]{\kappa^{*}_{#1}[#3+{}^{#2}#3,\kappa^L_{#2}(#4,#5)]}
\newcommand{\phibulletxy}[5]{\kappaLxy{#1}{#2}{#3}{#4}{#5}+\kappaLxy{#1}{#2}{#4}{#5}{#3}+\kappaLxy{#1}{#2}{#5}{#3}{#4}}
\newcommand{\kappaLLxy}[5]{\kappa^L_{#1}[#3+{}^{#2}#3,\kappa^L_{#2}(#4,#5)]}
\newcommand{\phiLxy}[5]{\kappaLLxy{#1}{#2}{#3}{#4}{#5}+\kappaLLxy{#1}{#2}{#4}{#5}{#3}+\kappaLLxy{#1}{#2}{#5}{#3}{#4}}

\DeclareMathOperator{\HH}{HH}
\DeclareMathOperator{\Ext}{Ext}
\DeclareMathOperator{\Hom}{Hom}
\DeclareMathOperator{\Span}{Span}
\DeclareMathOperator{\Sym}{Sym}
\DeclareMathOperator{\Alt}{Alt}
\DeclareMathOperator{\im}{im}
\DeclareMathOperator{\codim}{codim}
\DeclareMathOperator{\gr}{gr}
\DeclareMathOperator{\tri}{3-cyc}
\DeclareMathOperator{\fivecyc}{5-cyc}
\DeclareMathOperator{\penta}{5-cyc}

\theoremstyle{plain}
\newtheorem{theorem}{Theorem}[section]
\newtheorem*{theorem*}{Theorem}
\newtheorem{proposition}[theorem]{Proposition}
\newtheorem{lemma}[theorem]{Lemma} 
\newtheorem{corollary}[theorem]{Corollary}

\theoremstyle{definition}
\newtheorem{example}[theorem]{Example}
\newtheorem{remark}[theorem]{Remark}
\newtheorem{definition}[theorem]{Definition}
\newtheorem{case}{Case}  

\numberwithin{equation}{section}

\begin{document}

\begin{abstract}
Drinfeld orbifold algebras are a type of
deformation of skew group algebras generalizing graded Hecke algebras of interest in
representation theory, algebraic combinatorics, and noncommutative geometry. 
In this article, we classify all Drinfeld orbifold algebras for symmetric groups acting by the natural permutation representation.  This provides, for nonabelian groups, infinite families of examples of Drinfeld orbifold algebras that are not graded Hecke algebras.  We include explicit descriptions of the maps recording commutator relations and show there is a one-parameter family of such maps supported only on the identity and a three-parameter family of maps supported only on $3$-cycles and $5$-cycles.
Each commutator map must satisfy properties arising from a Poincar\'{e}-Birkhoff-Witt condition on the algebra, and our analysis of the properties illustrates reduction techniques using orbits of group element factorizations and intersections of fixed point spaces.
\end{abstract}

\title[Orbifold algebras for $S_n$]
{Drinfeld orbifold algebras for symmetric groups}

\author{B.\ Foster-Greenwood}
\address{Department of Mathematics and Statistics, California State Polytechnic University,
Pomona, California 91768, USA}\email{brianaf@cpp.edu}
\author{C.\ Kriloff}
\address{Department of Mathematics and Statistics, Idaho State University,
Pocatello, Idaho 83209, USA}\email{krilcath@isu.edu}


\maketitle

\section{Introduction}
Numerous algebras of intense recent study and interest arise as deformations of skew group algebras $S(V)\# G$, where $G$ is a finite group acting linearly on a finite-dimensional vector space $V$ and $S(V)$ is the symmetric algebra.
A grading on the skew group algebra is determined by assigning
degree one to vectors in $V$ and degree zero to elements of the group algebra. Drinfeld graded Hecke algebras are constructed by identifying commutators of
elements of $V$ with carefully chosen elements of degree zero (i.e., from the group algebra) to
yield a deformation of the skew group algebra.  In~\cite{SWorbifold}, Drinfeld orbifold algebras are similarly defined but additionally allow for degree-one terms in the commutator relations.  The resulting algebras are also
deformations of the skew group algebra.

Besides capturing a new realm of deformations of skew group algebras, Drinfeld orbifold algebras encompass many known algebras of interest in representation theory, noncommutative geometry, and mathematical physics.  The
term ``Drinfeld orbifold algebras'' alludes to the subject's origins in~\cite{Drinfeld1986}, where Drinfeld introduced a broad class of algebras to serve as noncommutative coordinate rings for singular orbifolds.
When the group is a Coxeter group acting by its reflection representation, Drinfeld's algebras are isomorphic (see~\cite{RamShepler})
to the graded Hecke algebras from~\cite{Lusztig1988}, which arise from a filtration of an affine Hecke algebra when the group is crystallographic (see~\cite{Lusztig1989}).  The representation theory of these algebras is useful in understanding representations and geometric structure of reductive $p$-adic groups.

A recent focus on symplectic reflection algebras, which are Drinfeld Hecke algebras for symplectic reflection groups acting on a symplectic vector space, began with~\cite{EtingofGinzburg2002}.
The importance of these algebras lies in the fact that the center of the skew group algebra is the ring of invariants, $\CC[V]^G=\mathrm{Spec}(V/G)$, and in the philosophy that the center of a deformation of the skew group algebra may then deform $\CC[V]^G$ (see the surveys~\cite{Gordon2008,Bellamy2016}).
As a special case, rational Cherednik algebras arise by pairing a reflection representation with its dual and are related to
integrable Calogero-Moser systems in physics and deep results in combinatorics (see for instance the surveys~\cite{Gordon2010,Etingof2014}).

Drinfeld orbifold algebras afford two advantageous views: as quotient algebras satisfying a Poincar\'{e}-Birkhoff-Witt (PBW) condition and as formal algebraic deformations of skew group algebras.  While
PBW conditions relate an algebra to homogeneous shadows of itself that have well-behaved bases, algebraic deformation theory (\`{a} la Gerstenhaber~\cite{GerstenhaberSchack}) focuses on how the multiplicative structure varies with a deformation parameter and provides a framework of understanding via Hochschild cohomology. In particular, every formal deformation arises from a Hochschild 2-cocycle.

Fruitful techniques arise from a melding of
the PBW perspective with the deformation theory perspective (see the survey~\cite{SWPBWsurvey2015}).
Braverman and Gaitsgory~\cite{BravermanGaitsgory} and also Polishchuk and Positelski~\cite{PolishchukPositselski} initiated the use of homological methods to study PBW conditions in the context of quadratic algebras of Koszul type.  Etingof and Ginzburg applied some of these ideas in an expanded setting in their seminal paper on symplectic reflection algebras.
The study of Drinfeld orbifold algebras also benefits from relating PBW conditions to
formal deformations. 
Shepler and Witherspoon prove two characterizations of Drinfeld orbifold
algebras: a concrete ring theoretic version~\cite[Theorem~3.1]{SWorbifold} (proved using Composition-Diamond Lemmas and Groebner basis theory)  and a cohomological version~\cite[Theorem~7.2]{SWorbifold}.

In the present case study, we classify Drinfeld orbifold algebras for symmetric groups acting by the natural permutation representation.  In Section~\ref{sec:cohomology}, we apply~\cite[Theorem~7.2]{SWorbifold} and use Hochschild cohomology to find possible degree-one terms of the commutator relations for a Drinfeld orbifold algebra.  In Section~\ref{sec:computations}, we then work with~\cite[Theorem~3.1]{SWorbifold} to determine compatible degree-zero terms (if they exist).
Our main result, stated in Theorems~\ref{thm:LieOrbifoldAlgebraMaps} and~\ref{thm:DrinfeldOrbifoldAlgebraMaps}, is an explicit description of the parameter maps that define Drinfeld orbifold algebras for symmetric groups. 

Parameter maps of Drinfeld orbifold algebras record commutators of elements of the vector space $V$ and can be categorized
based on their support, i.e., which group elements appear in the image.
Drinfeld orbifold algebra maps (see Definition~\ref{def:fourconditions}) with their linear part supported only on the identity give rise to Lie orbifold algebras, as defined in~\cite{SWorbifold}.  Lie orbifold algebras generalize universal enveloping algebras of Lie algebras, just as symplectic reflection algebras generalize Weyl algebras.
We summarize our results classifying Lie and Drinfeld orbifold algebras.  
\begin{theorem*} 
	For the symmetric group $S_n$ ($n\geq3$) acting on $V\cong\CC^n$ by the natural
	permutation representation, there is a one-parameter family of Lie orbifold algebras. 
\end{theorem*}
The remaining algebras have commutator relations supported only on $3$-cycles and $5$-cycles.  
\begin{theorem*}
	For the symmetric group $S_n$ ($n\geq 4$) acting on $V\cong\CC^n$ by the natural permutation representation, there is a three-parameter family of Drinfeld orbifold algebras supported on $3$-cycles and $5$-cycles.  For $n=3$, the family involves only two parameters. 
\end{theorem*} 

In Section~\ref{sec:OrbifoldAlgebras}, we present the algebras via generators and relations.  The examples contribute to an expanding medley of ``degree-one deformations''.  For instance, Shakalli~\cite{Shakalli} uses actions of Hopf algebras to construct examples of deformations of quantum skew group algebras involving degree-one terms in the commutator relations.  Shepler and Witherspoon~\cite{SWorbifold} consider Drinfeld orbifold algebras for groups acting diagonally.  The algebras we construct are among the first examples of Drinfeld orbifold (but not Hecke) algebras for nonabelian groups.

A fundamental problem in deformation theory is to determine which Hochschild
$2$-cocycles actually lift to deformations.  The results of Section~\ref{sec:mainresult} provide a family of $2$-cocycles that lift to define Drinfeld orbifold algebras for symmetric groups.
However, we also show, in Proposition~\ref{prop:L1L3incompatible},
that for symmetric groups, degree-one Hochschild $2$-cocycles simultaneously supported on and off the identity do not lift to yield Drinfeld orbifold algebras (and in fact do not even define Poisson structures).  This contrasts with the Drinfeld Hecke algebra case in which every polynomial degree-zero Hochschild $2$-cocycle determines a deformation of the skew group algebra.

Reduction techniques in Section~\ref{sec:reduction} and simplifications in Section~\ref{sec:computations} may prove helpful in predicting for which group actions and spaces candidate cocycles will lift to yield Drinfeld orbifold algebras.
In particular, a variation of Lemma~\ref{le:orbitreduction} may be effective for other groups with centralizers acting by monomial matrices, and the pattern to the values in Lemmas~\ref{lemma:simplification} and~\ref{le:simplification53} might generalize
to other group representations through analysis of intersections of fixed point spaces.
As further exploration, one could consider Drinfeld
orbifold algebras for symmetric groups in the twisted or quantum settings, as has been done for Drinfeld Hecke algebras for symmetric groups in~\cite[Example~2.17]{Witherspoon2007} and~\cite[Theorem~6.9]{Naidu2016}.

\section{Preliminaries}\label{sec:preliminaries}

Throughout, we let $G$ be a finite group acting linearly on a vector space $V\cong\CC^n$. All tensors will be over $\CC$. 

\subsection*{Skew group algebras}
Let $G$ be a finite group that acts on a $\CC$-algebra $R$ by
algebra automorphisms, and write $\actby{g}s$ for the result of acting by $g\in G$ on $s\in R$.  The \standout{skew group algebra} $R\#G$ is the
semi-direct product algebra $R\rtimes \CC G$ with underlying vector space $R\otimes \CC G$ and multiplication of simple tensors defined by
$$
(r\otimes g)(s\otimes h)=r(\actby{g}s)\otimes gh
$$
for all $r,s\in R$ and $g,h\in G$.
The skew group algebra becomes a $G$-module by letting $G$ act diagonally on
$R\otimes \CC G$, with conjugation on the group algebra factor: 
$$
\actby{g}(s\otimes h)=(\actby{g}s)\otimes(\actby{g}h)=(\actby{g}s)\otimes ghg^{-1}.
$$
In working with elements of skew group algebras, we commonly omit tensor symbols unless
the tensor factors are lengthy expressions. 

If $G$ acts linearly on
a vector space $V\cong\CC^n$, then $G$ also acts on the tensor algebra $T(V)$ and
symmetric algebra $S(V)$ by algebra automorphisms.
The skew group algebras $T(V)\# G$ and $S(V)\# G$ become graded
algebras when elements of $V$ are assigned degree one and elements of $G$ are assigned degree zero.

\subsection*{Cochains} 
A \standout{$k$-cochain} is a $G$-graded linear map $\alpha=\sum_{g\in G}\alpha_g g$ with components $\alpha_g:\bigwedge^k V\to S(V)$.
 (Details in Section~\ref{sec:cohomology} motivate the use of cohomological terminology.)
 If each $\alpha_g$ maps into $V$, then $\alpha$ is called a 
 \standout{linear cochain}, and if each $\alpha_g$ maps into $\CC$, then $\alpha$ is called
 a \standout{constant cochain}.  

We regard a map $\alpha$ on $\bigwedge^kV$ as a multilinear alternating map on $V^k$ and write $\alpha(v_1,\ldots,v_k)$ in place of $\alpha(v_1\wedge\cdots\wedge v_k)$. 
     Of course, if 
$\alpha(v_1,\ldots,v_k)=0$, then $\alpha$ is zero on any permutation of
$v_1,\ldots,v_k$. Also, if $\alpha$ is zero on all $k$-tuples of basis vectors, then $\alpha$ is zero on any $k$-tuple of vectors. We exploit these facts often in the computations in Section~\ref{sec:computations}.

The \standout{support of a cochain $\alpha$} is the set of group elements for which the component $\alpha_g$ is not the zero map. The \standout{kernel of a cochain $\alpha$} is the set of vectors $v_0$ such that $\alpha(v_0,v_1,\ldots,v_{k-1})=0$ for all $v_1,\ldots,v_{k-1}\in V$.

The group $G$ acts on the components of a cochain.  Specifically, for a group element $h$ and component $\alpha_g$, the map
$\actby{h}\alpha_g$ is defined by $(\actby{h}\alpha_g)(v_1,\ldots,v_k)=\actby{h}(\alpha_g(\actby{h^{-1}}v_1,\ldots,\actby{h^{-1}}v_k))$. In turn, the group acts on the space of cochains by letting $\actby{h}\alpha=\sum_{g\in G}\actby{h}\alpha_g\otimes hgh^{-1}$. Thus $\alpha$ is a \standout{$G$-invariant cochain} if and only if $\actby{h}\alpha_g=\alpha_{hgh^{-1}}$ for all $g,h\in G$.

\subsection*{Drinfeld orbifold algebras}
For a parameter map $\kappa=\kappa^L+\kappa^C$, where $\kappa^L$ is a linear $2$-cochain and $\kappa^C$ is a constant $2$-cochain, the quotient algebra
$$
\mathcal{H}_{\kappa}=T(V)\#G/\langle vw-wv-\kappa^L(v,w)-\kappa^C(v,w)\mid v,w\in V\rangle
$$
is called a \standout{Drinfeld orbifold algebra} if the associated graded algebra $\gr\mathcal{H}_{\kappa}$ is isomorphic to the skew group algebra
$S(V)\# G$.  The condition $\gr\mathcal{H}_{\kappa}\cong S(V)\# G$ is
called a \standout{Poincar\'{e}-Birkhoff-Witt (PBW) condition}, in analogy with
the PBW Theorem for universal enveloping algebras.

Further, if $\mathcal{H}_{\kappa}$ is a Drinfeld orbifold algebra and $t$ is a complex parameter,
then 
$$
\mathcal{H}_{\kappa,t}:=T(V)\#G[t]/\langle vw-wv-\kappa^L(v,w)t-\kappa^C(v,w)t^2\mid v,w\in V\rangle
$$
is called a \standout{Drinfeld orbifold algebra over $\CC[t]$}.
In~\cite[Theorem~2.1]{SWorbifold}, Shepler and Witherspoon make an explicit
connection between the PBW condition and deformations in the sense of Gerstenhaber~\cite{GerstenhaberSchack} by showing how to interpret Drinfeld orbifold
algebras over $\CC[t]$ as formal deformations of the skew group algebra
$S(V)\# G$.  We summarize the broader context of formal deformations in Section~\ref{sec:cohomology}.

\subsection*{Drinfeld orbifold algebra maps}
Though the defining PBW condition for a Drinfeld orbifold algebra $\mathcal{H}_{\kappa}$ involves an isomorphism of algebras, Shepler and Witherspoon proved an equivalent characterization~\cite[Theorem~3.1]{SWorbifold} in terms of properties of the parameter map $\kappa$.

\begin{definition}\label{def:fourconditions}
	Let $\kappa=\kappa^L+\kappa^C$ where $\kappa^L$ is a linear $2$-cochain and $\kappa^C$ is a constant $2$-cochain, and let $\Alt_3$ denote the alternating group on three elements.
	We say $\kappa$ is a {\bf Drinfeld orbifold algebra map} if the following conditions are satisfied for all $g\in G$ and $v_1,v_2,v_3\in V$:
	\addtocounter{equation}{-1} 
	\begin{equation}\label{def:fourconditions-0}
		 \im\kappa^L_g\subseteq V^g,
	\end{equation}
	\begin{equation}\label{def:fourconditions-i}
		\text{the map $\kappa$ is $G$-invariant,}
	\end{equation}
	\begin{equation}\label{def:fourconditions-ii}
		\sum_{\sigma\in\Alt_3}\kappa_g^L(v_{\sigma(2)},v_{\sigma(3)})(\actby{g}v_{\sigma(1)}-v_{\sigma(1)})=0\text{ in $S(V)$,}
	\end{equation}
	\begin{equation}\label{def:fourconditions-iii}
		\sum_{\sigma\in\Alt_3}\sum_{xy=g}
		\kappa_{x}^L(v_{\sigma(1)}+\actby{y}v_{\sigma(1)},\kappa_y^L(v_{\sigma(2)},v_{\sigma(3)}))
		=2\sum_{\sigma\in\Alt_3}\kappa_g^C(v_{\sigma(2)},v_{\sigma(3)})(\actby{g}v_{\sigma(1)}-v_{\sigma(1)}),
	\end{equation}
	\begin{equation}\label{def:fourconditions-iv}
		\sum_{\sigma\in\Alt_3}\sum_{xy=g}
		\kappa_{x}^C(v_{\sigma(1)}+\actby{y}v_{\sigma(1)},\kappa_y^L(v_{\sigma(2)},v_{\sigma(3)}))=0.
	\end{equation}
	As a special case, if the linear component $\kappa^L$ of
	a Drinfeld orbifold algebra map is supported only on the identity, then we also call $\kappa$ a \standout{Lie orbifold algebra map}.
\end{definition}

\begin{remark}\label{remark:image}
	If $\mathcal{H}_{\kappa}$ is a Drinfeld orbifold algebra, then $\kappa$ must satisfy conditions~(\ref{def:fourconditions-i})-(\ref{def:fourconditions-iv}), but not necessarily the image constraint~(\ref{def:fourconditions-0}).  
	However,~\cite[Theorem~7.2 (ii)]{SWorbifold} guarantees there will exist a 
	Drinfeld orbifold algebra $\mathcal{H}_{\widetilde{\kappa}}$ such that
	$\mathcal{H}_{\widetilde{\kappa}}\cong\mathcal{H}_{\kappa}$ as
	filtered algebras and $\widetilde{\kappa}$ satisfies the image constraint
	$\im \widetilde{\kappa}^L_g\subseteq V^g$ for each $g$ in $G$.
	Thus, in classifying Drinfeld orbifold algebras, it suffices to only 
	consider Drinfeld orbifold algebra maps.
\end{remark}

\begin{theorem}[{\cite[Theorem~3.1 and Theorem~7.2 (ii)]{SWorbifold}}] 
	A quotient algebra $\mathcal{H}_{\kappa}$ satisfies the PBW condition
	$\gr\mathcal{H}_{\kappa}\cong S(V)\# G$ if and only if there exists a
	Drinfeld orbifold algebra map $\widetilde{\kappa}$ such that $\mathcal{H_{\kappa}}\cong\mathcal{H_{\widetilde{\kappa}}}$.
\end{theorem}

The process of determining the set of all Drinfeld orbifold algebra maps 
consists of two phases.
For reasons discussed in Section~\ref{sec:cohomology}, we use language from
cohomology and deformation theory to describe each phase.  First, one finds
all \standout{pre-Drinfeld orbifold algebra maps}, i.e., all $G$-invariant linear $2$-cochains $\kappa^L$ satisfying Properties~(\ref{def:fourconditions-0}) and (\ref{def:fourconditions-ii}).
A bijection between pre-Drinfeld orbifold algebra maps
and a particular set of representatives of Hochschild cohomology classes
facilitates this step (see Lemma~\ref{ineedaname}).
Second, we determine for which pre-Drinfeld orbifold algebra maps $\kappa^L$ there exists a compatible $G$-invariant constant
$2$-cochain $\kappa^C$ such that Properties~(\ref{def:fourconditions-iii}) and (\ref{def:fourconditions-iv}) hold.
We say $\kappa^C$ \standout{clears the first obstruction} if Property~(\ref{def:fourconditions-iii}) holds
and \standout{clears the second obstruction} if Property~(\ref{def:fourconditions-iv}) holds.
If a $G$-invariant constant $2$-cochain $\kappa^C$ clears both obstructions, then
we say $\kappa^L$ \standout{lifts} to the Drinfeld orbifold algebra map $\kappa=\kappa^L+\kappa^C$.

\section{Orbifold algebras for symmetric groups}\label{sec:OrbifoldAlgebras}

Let $e_1,\ldots,e_n$ be the standard basis of $V\cong\CC^n$.
Let the symmetric group $S_n$ act on $V$ by its natural permutation
representation, so $\actby{\sigma}e_i=e_{\sigma(i)}$ for $\sigma$ in $S_n$.
  The main effort of this paper is in proving Theorems~\ref{thm:LieOrbifoldAlgebraMaps} and~\ref{thm:DrinfeldOrbifoldAlgebraMaps}, which describe all Drinfeld orbifold algebra maps for $S_n$ acting by the natural permutation representation.  As corollaries of the theorems in Section~\ref{sec:mainresult}, we present here the resulting PBW deformations of the skew group algebra $S(V)\# S_n$ via generators and relations.

First, the one-dimensional space of Lie orbifold algebra maps classified
in Theorem~\ref{thm:LieOrbifoldAlgebraMaps} yields a family of Lie orbifold
algebras arising as deformations of $S(V)\# S_n$.

\begin{theorem}[{Lie Orbifold Algebras over $\CC[t]$}]\label{Lieorbifold} 
Let the symmetric group $S_n$ ($n\geq3$) act on $V\cong\CC^n$ by
its natural permutation representation. Then for $a\in\CC$,
$$
\mathcal{H}_{\kappa,t}=T(V)\# S_n[t]/\langle e_ie_j-e_je_i-a(e_i-e_j)t\mid 1 \leq i < j\leq n\rangle
$$
is a Lie orbifold algebra over $\CC[t]$.  Further,
the algebras $\mathcal{H}_{\kappa,1}$ are precisely the Drinfeld orbifold algebras such that $\kappa^L$ is supported only on the identity.
\end{theorem}

Second, in Theorem~\ref{thm:DrinfeldOrbifoldAlgebraMaps}, we determine (for $S_n$) all Drinfeld orbifold algebra maps such that the linear component $\kappa^L$ is
supported only off the identity. 
The relations in the consequent PBW deformations of $S(V)\# S_n$ involve sums of basis vectors over certain subsets of $[n]:=\{1,\ldots,n\}$.  For $I\subseteq[n]$, 
let $e_I=\sum_{i\in I}e_i$, 
and let $e_I^{\perp}$ denote the complementary vector $e_{[n]}-e_I$.

\begin{theorem} \label{mainexamples}
    Let the symmetric group $S_n$ ($n\geq3$) act on $V\cong\CC^n$
    by its natural permutation representation. For $a,b,c\in\CC$ and
    $1\leq i<j\leq n$ let
    $$
    \kappa^L(e_i,e_j)=\sum_{k\neq i,j} (ae_{\{i,j,k\}}+b{e^{\perp}_{\{i,j,k\}}})\otimes ((ijk)-(kji) ),
    $$
    and let
    $$
    \kappa^C(e_i,e_j)=c\sum_{k\neq i,j}((ijk)-(kji) )+ (a-b)^2\left(\sum_{\substack{\sigma\text{ a $5$-cycle}\\ \sigma^2(i)=j}}\hspace{-10pt}2(\sigma-\sigma^{-1})
    -\sum_{\substack{\sigma\text{ a $5$-cycle}\\ \sigma(i)=j}}\hspace{-10pt}(\sigma-\sigma^{-1})\right).
    $$
    Then
	$$
	\mathcal{H}_{\kappa,t}=T(V)\# S_n[t]/\langle e_ie_j-e_je_i-\kappa^L(e_i,e_j)t-\kappa^C(e_i,e_j)t^2\mid 1\leq i<j\leq n\rangle
	$$
	is a Drinfeld orbifold algebra over $\CC[t]$.  Further,
	the algebras $\mathcal{H}_{\kappa,1}$ are precisely the Drinfeld orbifold algebras such that $\im\kappa_g^L\subseteq V^g$ for each $g\in S_n$ and
	$\kappa^L$ is supported only off the identity.
\end{theorem}

We illustrate Theorem~\ref{mainexamples} for some small values of $n$.
Note that the parameter $b$ is irrelevant when $n=3$,
and the sums over $5$-cycles are absent in the cases $n=3$ and $n=4$.

\begin{example}
	For the symmetric group $S_3$ acting on $V\cong\CC^3$ by the natural permutation representation, the Drinfeld orbifold algebras such that
	$\im\kappa_g^L\subseteq V^g$ for each $g\in S_3$ and $\kappa_1^L=0$ are the algebras of
	the form
	$$
	\mathcal{H}_{\kappa}=T(V)\# S_3/\langle e_ie_j-e_je_i-\kappa(e_i,e_j)\mid 1\leq i<j\leq 3\}\rangle,
	$$
	where for some $a,c\in \CC$
  	$$
  	\kappa(e_1,e_2)=\kappa(e_2,e_3)=\kappa(e_3,e_1)=(a(e_1+e_2+e_3)+c)\otimes((123)-(321)).
  	$$
  	This example coincides with~\cite[Example 3.4]{SWorbifold} with a change of basis.
\end{example}

\begin{example} For the symmetric group $S_4$ acting on $V\cong\CC^4$ by the natural permutation representation, the Drinfeld orbifold
	algebras such that $\im\kappa_g^L\subseteq V^g$ for each $g\in S_4$ and $\kappa_1^L=0$ are
	the algebras of the form
	$$
	\mathcal{H}_{\kappa}=T(V)\# S_4/\langle e_ie_j-e_je_i-\kappa(e_i,e_j)\mid 1\leq i<j\leq 4\}\rangle,
	$$
	where
	\begin{alignat*}{2}
		\kappa(e_1,e_2) &= &(a(e_1+e_2+e_3)+be_4+c)\otimes((123)-(321))\\
		&&+(a(e_1+e_2+e_4)+be_3+c)\otimes((124)-(421)),
	\end{alignat*}
	and $\kappa(e_{\sigma(1)},e_{\sigma(2)})=\actby{\sigma}(\kappa(e_1,e_2))$ 
	for $\sigma$ in $S_4$.  (In acting by $\sigma$, recall that $\actby{\sigma}e_i=e_{\sigma(i)}$ and
	$\actby{\sigma}\tau=\sigma\tau\sigma^{-1}$ for $\sigma,\tau\in S_n$.)
\end{example}

\begin{example} For the symmetric group $S_5$ acting on $V\cong\CC^5$ by the natural permutation representation, the Drinfeld orbifold
	algebras such that $\im\kappa_g^L\subseteq V^g$ for each $g\in S_5$ and $\kappa_1^L=0$ are
	the algebras of the form
	$$
	\mathcal{H}_{\kappa}=T(V)\# S_5/\langle e_ie_j-e_je_i-\kappa(e_i,e_j)\mid 1\leq i<j\leq 5\}\rangle,
	$$
	where
\begin{align*}
	\kappa(e_1,e_2) &= &&\phantom{mm}(a(e_1+e_2+e_3)+b(e_4+e_5)+c)\otimes ((123)-(321))\\
	& &&+\phantom{2}(a(e_1+e_2+e_4)+b(e_5+e_3)+c)\otimes ((124)-(421)) \\
	& &&+\phantom{2}(a(e_1+e_2+e_5)+b(e_3+e_4)+c)\otimes ((125)-(521)) \\
	& &&-\phantom{2}(a-b)^2 \otimes ((12345)+(12543)+(12453)+(12354)+(12534)+(12435)) \\
	& &&+\phantom{2}(a-b)^2 \otimes ((21345)+(21543)+(21453)+(21354)+(21534)+(21435)) \\
	& &&-2(a-b)^2 \otimes ((23145)+(25143)+(24153)+(23154)+(25134)+(24135)) \\
	& &&+2(a-b)^2 \otimes ((13245)+(15243)+(14253)+(13254)+(15234)+(14235)),
\end{align*}
and $\kappa(e_{\sigma(1)},e_{\sigma(2)})=\actby{\sigma}(\kappa(e_1,e_2))$ 
for $\sigma$ in $S_5$.
\end{example}

\begin{remark} If we specialize to $t=1$ and let $a=b=0$ in Theorem~\ref{mainexamples}, then the linear component $\kappa^L$ is identically zero, thus recovering Drinfeld graded Hecke algebras for $S_n$.
\end{remark}

\section{Deformation algebras and Hochschild cohomology}\label{sec:cohomology}

Our goal in this section is to describe linear and constant
$2$-cochains $\kappa$ that are $G$-invariant and satisfy the \standout{mixed Jacobi identity}$$
[v_1,\kappa(v_2,v_3)]+[v_2,\kappa(v_3,v_1)]+[v_3,\kappa(v_1,v_2)]=0\quad\text{in $S(V)\# G$.}
$$
When $\kappa$ is expanded as $\sum_{g\in G}\kappa_g g$, it becomes clear that the
mixed Jacobi identity for $\kappa^L$ is equivalent to Property~(\ref{def:fourconditions-ii}) of a Drinfeld orbifold algebra map, and the mixed Jacobi identity for $\kappa^C$ is equivalent to
Property~(\ref{def:fourconditions-iii}) in the special case that the left side of (\ref{def:fourconditions-iii}) is zero.
In light of the relation between Drinfeld orbifold algebras and formal deformations, Hochschild cohomology becomes a tool to facilitate finding Drinfeld orbifold algebra maps,
as summarized in Lemma~\ref{ineedaname}. 
We first review some background on
deformation theory and cohomology before turning to the specific case
of the symmetric groups.

\subsection*{Deformations and Hochschild cohomology}
Let $A$ be an algebra over $\CC$.  
For a complex parameter $t$, a \standout{deformation over $\CC[t]$} 
of $A$ is the vector space $A[t]$ with an associative multiplication $*$,
which is $\CC[t]$-bilinear and for $a,b$ in $A$ is recorded in the form 
$$
a * b = ab + \mu_1(a\otimes b)t + \mu_2(a\otimes b)t^2 +\cdots
$$
for some maps $\mu_i:A\otimes A\rightarrow A$ with the sum finite for each pair $a,b$.  Identifying coefficients on
$t^i$ in the expressions $a*(b*c)$ and $(a*b)*c$ yields
a cohomological relation involving the maps $\mu_1,\ldots,\mu_i$.
For example, identifying coefficients of $t$ shows that $\mu_1$ is a \standout{Hochschild $2$-cocycle},
and identifying coefficients of $t^2$ shows the
(Hochschild) coboundary of $\mu_2$ must be half of the Gerstenhaber bracket of $\mu_1$
with itself.

Generally, for an $A$-bimodule $M$ the \standout{Hochschild cohomology}
of $A$ with coefficients in $M$ is $\HH^{\mol}(A,M):=\Ext_{A\otimes A^{\text{op}}}^{\mol}(A,M)$, and if $M=A$, we simply write $\HH^{\mol}(A)$.
Hochschild cohomology may be computed using various resolutions, each with their 
own advantages.   The maps $\mu_i$ defining
the multiplication of a formal deformation algebra are most easily regarded
as cochains on a bar resolution.  
However, when $A$ is a skew group algebra, advantageous formulations 
of Hochschild cohomology arise from a Koszul resolution and frame cohomology in 
terms of invariant theory.
Conversions between the bar complex and Koszul complex are key to the proof of
\cite[Theorem~2.1]{SWorbifold} that shows how to interpret a Drinfeld orbifold algebra over $\CC[t]$ as a formal deformation of a skew group algebra. 
The parameter map $\kappa$ of a Drinfeld orbifold algebra $\mathcal{H}_{\kappa,t}$ over $\CC[t]$ may be identified with a cochain on the
Koszul complex, and the linear part $\kappa^L$ relates to the first multiplication map $\mu_1$, 
while the constant part $\kappa^C$ relates to the second multiplication
map $\mu_2$ (see~\cite[Remark~2.5]{SWorbifold}). 

Cohomological relations involving the maps $\mu_i$ 
have implications for the components of the parameter map $\kappa$.
Indeed, the conditions on $\kappa$ given in~\cite[Theorem~3.1]{SWorbifold}
have a parallel statement~\cite[Theorem~7.2]{SWorbifold} 
in terms of cohomological spaces and operations.  While Properties~(\ref{def:fourconditions-i}) and (\ref{def:fourconditions-ii}) of a 
Drinfeld orbifold algebra map are stated with minimal machinery, the cohomological 
interpretations aid in organizing computations and also reveal some hidden implications emphasized in Remark~\ref{remark:codim}.

We now record descriptions of the cohomological spaces we use in our
computations and refer the reader to~\cite{SWorbifold}, for example, for more details on the bar and Koszul resolutions, chain maps, and isomorphisms that lead to these spaces. 
Let $G$ be a finite group acting linearly on a vector space $V\cong\CC^n$.
Let $H^{\mol}$ be the $G$-graded vector space $H^{\mol}=\bigoplus_{g\in G}H^{\mol}_g$
with components
$$
H^{p,d}_g=
S^d(V^{g}) \otimes 
\bigwedge^{p-\codim(V^g)}(V^g)^{*} \otimes
\bigwedge^{\codim(V^g)}\bigl((V^g)^*\bigr)^{\perp} \otimes
\mathbb{C}g,
$$
where $V^g$ is the fixed point space of $g$. Thus $H^{\mol}$ is tri-graded by
cohomological degree $p$, homogeneous polynomial degree $d$, and group element $g$.
For any set $R$ carrying a $G$-action, we write $R^G$ for the set of elements fixed
by every $g$ in $G$.
With the group $G$ acting diagonally on the tensor product (and with conjugation on
the group algebra factor),
the Hochschild cohomology of $S(V)\# G$ can be computed using the
series of isomorphisms
$$
\HH^{\mol}(S(V)\#G)
\cong\HHSSG{\mol}^G
\cong (H^{\mol})^G.
$$
The first isomorphism follows from \c{S}tefan~\cite{Stefan} (for example),
and the description of $H^{\mol}$ was first given independently by
Farinati~\cite{Farinati} and by Ginzburg-Kaledin~\cite{GinzburgKaledin}.

Note that, together, the exterior factors of $H_g^{p,d}$ identify with a
subspace of $\bigwedge^pV^*$, and then, since $S^d(V^g)\otimes\bigwedge^pV^*\otimes\CC g\cong\Hom(\bigwedge^p V,S^d(V^g)g)$, the space $H^{\mol}$ may be identified with a subspace of the cochains introduced in Section~\ref{sec:preliminaries}. The next lemma 
records the relationship between Properties~(\ref{def:fourconditions-i}) and (\ref{def:fourconditions-ii}) of a Drinfeld orbifold algebra map and Hochschild cohomology.  When $d=1$, the lemma is a restatement of~\cite[Theorem~7.2 (i) and (ii)]{SWorbifold}. When $d=0$, the lemma is a restatement of~\cite[Corollary 8.17(ii)]{SheplerWitherspoon08}.
Despite its cohomological heritage, it is also possible to give a 
linear algebraic proof of Lemma~\ref{ineedaname} 
in the spirit of~\cite[Lemma~1.8]{RamShepler}.

\begin{lemma}\label{ineedaname}
	For a $2$-cochain $\kappa=\sum_{g\in G}\kappa_g g$ with $\im \kappa_g\subseteq S^d(V^g)$ for each $g\in G$, 
	the following are equivalent:
	\begin{enumerate}\setlength\itemsep{.5em}
		\item[(a)] The map $\kappa$ is $G$-invariant and satisfies
		the mixed Jacobi identity, i.e., for all $v_1,v_2,v_3\in V$
		$$
		[v_1,\kappa(v_2,v_3)]+[v_2,\kappa(v_3,v_1)]+[v_3,\kappa(v_1,v_2)]=0\quad\text{in $S(V)\# G$,}
		$$
		where $[\cdot,\cdot]$ denotes the commutator in $S(V)\# G$.
		\item[(b)] For all $g,h\in G$ and $v_1,v_2,v_3\in V$:
		\begin{enumerate}
			\item[(i)] $\actby{h}(\kappa_g(v_1,v_2))=\kappa_{hgh^{-1}}(\actby{h}v_1,\actby{h}v_2)$ and
			\item[(ii)] $\kappa_g(v_1,v_2)(\actby{g}v_3-v_3)
			+\kappa_g(v_2,v_3)(\actby{g}v_1-v_1)
			+\kappa_g(v_3,v_1)(\actby{g}v_2-v_2)=0$.
		\end{enumerate}
		\item[(c)] The map $\kappa$ is an element of
		$$
		(H^{2,d})^G=\left(\bigoplus_{g \in G}
		\Bigl(\:
		S^d(V^{g})g \otimes 
		\bigwedge^{2-\codim(V^g)}(V^g)^{*} \otimes
		\bigwedge^{\codim(V^g)}\bigl((V^g)^*\bigr)^{\perp}
		\:\Bigr)\right)^G.
		$$ 
	\end{enumerate}
\end{lemma}

\begin{remark}\label{remark:codim}
	Part (c) of Lemma~\ref{ineedaname} illuminates some hidden implications of parts (a) and (b).  For instance, $\kappa$ can only be supported on elements $g$
	with $\codim V^g\in\{0,2\}$, which is readily seen from part (c) by noting that
	negative exterior powers are zero and that an element $g$ with codimension one
	acts nontrivially on $H_g^{2,d}$.  
\end{remark}

In practice, one may simplify the computation of $(H^{\mol})^G$
by computing centralizer invariants for a set of conjugacy class representatives and then expanding into $G$-invariants.
Formally, $(H^{\mol})^G\cong\bigoplus_{g\in\mathscr{C}}(H_g^{\mol})^{Z(g)}$, where
$\mathscr{C}$ is a set of conjugacy class representatives, and $Z(g)$ is the
centralizer of $g$.
We review the explicit passage from a $Z(g)$-invariant to a $G$-invariant, which will be especially
relevant in translating the results of Lemma~\ref{le:tricandidate} into the maps in
Definition~\ref{def:kappatri}.
Recall that a cochain $\alpha=\sum_{g\in G}\alpha_g$ 
is $G$-invariant if and only if
$\actby{h}\alpha_g=\alpha_{hgh^{-1}}$ for all $g,h\in G$. Thus, if
$\alpha$ is a $G$-invariant cochain, then
$\alpha_g$ is $Z(g)$-invariant for each $g$, and $\alpha$ is determined by
its components for a set of conjugacy class representatives.
In particular, a centralizer invariant $\alpha_g$ extends uniquely to
a $G$-invariant element, supported on the conjugacy class of $g$, via the map
$\alpha_g\mapsto\sum_{h\in[G/Z(g)]}\actby{h}\alpha_g$, where $[G/Z(g)]$ is a set of left coset representatives.
Further, we have the following commutative diagram:
\begin{center}
	\begin{tikzpicture}[>=latex,scale=2]
	\draw [->] (-2,.75) -- (-2,.25);     
	\draw [->] (2,.75) -- (2,.25);     
	\draw [->] (-.75,0) -- (.25,0);   
	\draw [->] (-.75,1) -- (.25,1);   
	\draw[color=black] (-2,1) node { $\alpha_g\in(H_g^{2,d})^{Z(g)}$};
	\draw[color=black] (-2,0) node { $\kappa_g:\bigwedge^2 V\to S^d(V^g)$};
	\draw[color=black] (2,1) node { $\alpha\in (H^{2,d})^G$};
	\draw[color=black] (2,0) node { $\kappa:\bigwedge^2 V\to\ds\bigoplus_{h\in [G/Z(g)]}S^d(V^{hgh^{-1}})hgh^{-1}$.};
	\end{tikzpicture}
\end{center}
The vertical arrows are via the isomorphism
$S^d(V^g)g\otimes\bigwedge^2V^*\cong\Hom(\bigwedge^2 V,S^d(V^g)g)$.
The horizontal arrows are via the orbit-sum maps 
$$\alpha_g\mapsto \sum_{h\in[G/Z(g)]}\actby{h}\alpha_g
\qquad\text{and}\qquad
\kappa_g\mapsto\sum_{h\in[G/Z(g)]}\actby{h}\kappa_g\actby{h}g.$$

\subsection*{Hochschild cohomology for symmetric groups}

We now turn to the specific example of cohomology of skew group algebras
of symmetric groups with the natural permutation representation.
Much of the Hochschild cohomology of $S(V)\# S_n$ 
may be extracted as subcases of Hochschild cohomology for
skew group algebras of complex reflection groups $G(r,p,n)$
found in~\cite{SheplerWitherspoon08}. 
However, we provide computations for $S_n\cong G(1,1,n)$
here for purposes of self-containment and notational consistency.

For the remainder of the section, let $e_1,\ldots,e_n$ be the 
standard basis of $V\cong\CC^n$, and let the symmetric group $S_n$ act on $V$ by its natural permutation representation.  Thus for $\sigma$ in $S_n$, we have $\actby{\sigma}e_i=e_{\sigma(i)}$.

We first show that for the symmetric
group acting by its natural permutation representation, elements of
Hochschild $2$-cohomology with polynomial degree zero or one can only be
supported on the identity or $3$-cycles.

\begin{lemma}\label{le:support} 
	Let $S_n$ ($n\geq3$) act on $V\cong\CC^n$ by its natural permutation representation,
	and let $\alpha=\sum_{g\in S_n}\alpha_g$ be an element of $(H^{2,1}\oplus H^{2,0})^{S_n}$.
	If $g$ is not the identity and not a $3$-cycle, then $\alpha_g=0$.
\end{lemma}
\begin{proof}
	Let $\alpha=\sum_{g\in S_n}\alpha_g$ be an element of the cohomology space $(H^{2,1}\oplus H^{2,0})^{S_n}$.
	If $\codim(V^g)\not\in\{0,2\}$, then $\alpha_g=0$ by Remark~\ref{remark:codim}.  Under the permutation representation, the
	only elements with $\codim V^g\in\{0,2\}$ are the identity, the $3$-cycles, and the double-transpositions, so it remains to show that $\alpha_g= 0$ if $g$ is a double-transposition.  
	In fact, since $\alpha$ is determined by
	its components for a set of conjugacy class representatives, it
    suffices to show $(H_g^{2,1}\oplus H_g^{2,0})^{Z(g)}=0$ for $g=(12)(34)$.
	The vectors 
	\begin{eqnarray*}
		v_1 & = & e_1 - e_2 - e_3 + e_4, \\
		v_2 & = & e_1 - e_2 + e_3 - e_4, \\
		v_3 & = & e_1 + e_2 - e_3 - e_4, \\
		v_4 & = & e_1 + e_2 + e_3 + e_4, \\
		v_k & = & e_k \qquad\qquad\qquad\qquad\text{ for $5\leq k\leq n$}
	\end{eqnarray*} 
	form a $g$-eigenvector basis of $V$ with
	$(V^{g})^{\perp}=\Span\{v_1,v_2\}$ and $V^{g}=\Span\{v_3,\ldots,v_{n}\}$.
	The wedge product $v_1^*\wedge v_2^*$ is a scalar multiple of the
	volume form
	$$
	\vol{g}:=e_1^*\wedge e_3^*+e_3^*\wedge e_2^*+e_2^*\wedge e_4^*+e_4^*\wedge e_1^*,
	$$
	so $\vol{g}$ is a basis for $\bigwedge^2((V^g)^{\perp})^*$.
	The transposition $(12)$ commutes with $g$ but scales elements of
	$H_g^{2,1}\oplus H_g^{2,0}$ by negative one, so $(H_g^{2,1}\oplus H_g^{2,0})^{Z(g)}=0$.
\end{proof}

The cohomology elements in the next lemma give rise to 
the parameter maps of the Lie orbifold algebras exhibited in Theorem~\ref{Lieorbifold}.

\begin{lemma}\label{Lieorbifoldparametermaps}
	Let $G=S_n$ ($n\geq3$) act on $V\cong\CC^n$ by its natural permutation representation.
	The subspace of $(H^{2,1}\oplus H^{2,0})^{S_n}$ consisting of elements supported only
	on the identity is one-dimensional with basis $\sum_{1\leq i<j\leq n}(e_i-e_j)\otimes e_i^*\wedge e_j^*$. 
\end{lemma}
\begin{proof}
	Since the permutation representation of $S_n$ is a self-dual reflection representation, generators of 
	$(H_1^{\mol})^G\cong\left(S(V)\otimes \bigwedge V^*\right)^G\cong\left(S(V)\otimes \bigwedge V\right)^G$ are given by Solomon's Theorem from the
	invariant theory of reflection groups (see~\cite{Solomon}, or the expository~\cite[Chapter 22]{Kane}).
	Specifically, the power sums $f_k=e_1^k+\cdots+e_n^k$ with $1\leq k\leq n$ form a set
	of algebraically independent invariant polynomials, and the differential
	forms 
	$$
	\alpha_k:=\frac{1}{k+1}\sum_{i=1}^n\frac{\partial f_{k+1}}{\partial e_i}\otimes e_i^*=e_1^k\otimes e_1^*+\cdots+e_n^k\otimes e_n^*  \quad\text{for $0\leq k\leq n-1$}
	$$
	generate $\left(S(V)\otimes \bigwedge V^*\right)^G$ as an
	exterior algebra over $S(V)^G\cong\CC[f_1,\ldots,f_n]$.
	Thus, $(S(V)\otimes\bigwedge^2V^*)^G$ is freely generated as an $S(V)^G$-module by $\{\alpha_{k}\alpha_{l}\mid  0\leq k< l\leq n-1\}$.
	Since $\alpha_k\alpha_l$ has polynomial degree $k+l$, 
	every element of $(H_1^{2,1}\oplus H_1^{2,0})^G$ is a scalar multiple of $\alpha_1\alpha_0$.
\end{proof}

The cohomology in the next lemma serves two purposes.  The polynomial degree 
one elements give rise to pre-Drinfeld orbifold algebra maps supported only
on $3$-cycles, while the polynomial degree zero elements are needed
in constructing multiple liftings of a pre-Drinfeld orbifold algebra map.

\begin{lemma}\label{le:tricandidate}
	Let $S_n$ ($n\geq3$) act on $V\cong\CC^n$ by its natural permutation representation.
	The subspace of $(H^{2,1}\oplus H^{2,0})^{S_n}$ consisting of elements supported only on
	$3$-cycles is two-dimensional if $n=3$ and three-dimensional if $n\geq4$.
\end{lemma}
\begin{proof}
	Recall that a cohomology element is determined by its components for a set of conjugacy class representatives.
	Thus, if $\alpha$ is supported only on $3$-cycles, it suffices to choose a representative, say $g=(123)$,
	and find a basis of
		\[
		(H_g^{2,1}\oplus H_g^{2,0})^{Z(g)}
		\cong((V^g\oplus\CC)\otimes\textstyle\bigwedge^{2}\bigl((V^g)^{\perp}\bigr)^*\otimes\CC g)^{Z(g)}.
		\]
	Let $\omega=e^{2\pi i/3}$.  Then the vectors
	\begin{eqnarray*}
		v_1 & = & e_1 + \omega^2 e_2 + \omega e_3, \\
		v_2 & = & e_1 + \omega e_2 + \omega^2 e_3, \\
		v_3 & = & e_1 + e_2 + e_3, \\
		v_k & = & e_k \qquad\qquad\qquad\text{ for $4\leq k\leq n$} 
	\end{eqnarray*} 
	form a $g$-eigenvector basis of $V$ with $(V^{g})^{\perp}=\Span\{v_1,v_2\}$ and $V^{g}=\Span\{v_3,\ldots,v_{n}\}$.
	The wedge product $v_1^*\wedge v_2^*$ is a scalar multiple of 
	the volume form
	$$
	\vol{g}:=e_1^*\wedge e_2^*+e_2^*\wedge e_3^*+e_3^*\wedge e_1^*,
	$$
	so $\vol{g}$ is a basis for $\bigwedge^2 ((V^g)^{\perp})^*$. Each element of the centralizer
	$$
	Z(123)=\langle(123)\rangle\times\Sym_{\{4,\ldots,n\}}
	$$
	acts trivially on $\vol{g}$, so
	$$
	(H_g^{2,1}\oplus H_g^{2,0})^{Z(g)}
	\cong((V^g)^{Z(g)}\oplus\CC)\otimes\textstyle\bigwedge^{2}\bigl((V^g)^{\perp}\bigr)^*\otimes\CC g.
	$$
	The vectors $e_1+e_2+e_3$ and $e_4+\cdots+e_n$ 
	form a basis of $(V^g)^{Z(g)}$, so the cohomology elements
	$$
	\alpha_{(123)}=(e_1+e_2+e_3)\otimes\vol{(123)}\otimes(123)\quad\text{and}\quad
	\beta_{(123)}=(e_4+\cdots+e_n)\otimes\vol{(123)}\otimes(123)
	$$
	span $(H^{2,1}_{(123)})^{Z(g)}$, while $\gamma_{(123)}=\vol{g}\otimes (123)$ spans $(H_{(123)}^{2,0})^{Z(g)}$.  
\end{proof}

The following definition arises from expanding the centralizer invariants determined in the proof of Lemma~\ref{le:tricandidate} into $S_n$-invariants and applying the isomorphism $S^d(V^g)g\otimes\bigwedge^2V^*\cong\Hom(\bigwedge^2 V,S^d(V^g)g)$.
(See the discussion following Remark~\ref{remark:codim}.)

\begin{definition}\label{def:kappatri}
	For parameters $a,b\in \CC$, let 
	\(\kappa^L_{\tri}=\sum_{(ijk)\in S_n}\kappa^L_{(ijk)}\otimes(ijk)\) be the linear $2$-cochain
	with component maps $\kappa^L_{(ijk)}:\bigwedge^2 V\to V^{(ijk)}$ defined by
	$$
	\kappa^L_{(ijk)}(e_i,e_j)=\kappa^L_{(ijk)}(e_j,e_k)=\kappa^L_{(ijk)}(e_k,e_i)=a(e_i+e_j+e_k)+b\sum_{l\neq i,j,k}e_l
	$$
	and $\kappa^L_{\tri}(e_l,e_m)=0$ if $\{e_l,e_m\}\cap V^{(ijk)}\neq\varnothing$.
	
	For a parameter $c\in\CC$, let $\kappa^C_{\tri}=\sum_{(ijk)\in S_n}\kappa^C_{(ijk)}\otimes(ijk)$ be the constant cochain
	with component maps $\kappa^C_{(ijk)}:\bigwedge^2 V\to\CC$ defined by
	$$
	\kappa^C_{(ijk)}(e_i,e_j)=\kappa^C_{(ijk)}(e_j,e_k)=\kappa^C_{(ijk)}(e_k,e_i)=c
	$$
	and $\kappa^C_{\tri}(e_l,e_m)=0$ if $\{e_l,e_m\}\cap V^{(ijk)}\neq\varnothing$.
	
	Also let $\kappa_{\tri}=\kappa^L_{\tri}+\kappa^C_{\tri}$.
\end{definition}

Notice that if $a$ and $b$ are not both zero, then $\ker\kappa^L_{(ijk)}=V^{(ijk)}$,
and if $c\neq0$, then $\ker\kappa^C_{(ijk)}=V^{(ijk)}$.

In view of the equivalences in Lemma~\ref{ineedaname}, the polynomial degree one elements of Hochschild $2$-cohomology computed in Lemmas~\ref{le:support}, \ref{Lieorbifoldparametermaps}, and~\ref{le:tricandidate} yield
a description of all pre-Drinfeld orbifold algebra maps.

\begin{corollary}\label{cor:summarypreDOA}
	The pre-Drinfeld orbifold algebra maps for $S_n$ ($n\geq3$) acting by its natural permutation representation are the linear $2$-cochains $\kappa^L=\kappa_1^L+\kappa^L_{\tri}$, with $\kappa^L_{\tri}$ as in
	Definition~\ref{def:kappatri} and $\kappa_1^L(e_i,e_j)=a_1(e_i-e_j)$ for some $a_1\in\CC$.
\end{corollary}

In Theorems~\ref{thm:LieOrbifoldAlgebraMaps} and~\ref{thm:DrinfeldOrbifoldAlgebraMaps}, we will show that the maps
$\kappa_1^L$ and $\kappa^L_{\tri}$ lift (separately, but not in combination) to Drinfeld orbifold algebra maps.  Any two liftings of a particular pre-Drinfeld orbifold algebra map must differ by a constant $2$-cochain that satisfies the mixed Jacobi identity. Recalling Lemma~\ref{ineedaname}, the desired constant $2$-cochains are revealed by the polynomial degree zero elements of Hochschild $2$-cohomology computed in Lemmas~\ref{le:support}, \ref{Lieorbifoldparametermaps}, and~\ref{le:tricandidate}.

\begin{corollary}
\label{constantcocycles}
	The $S_n$-invariant constant $2$-cochains satisfying the
	mixed Jacobi identity are the maps $\kappa_{\tri}^C$ as in Definition~\ref{def:kappatri}.
\end{corollary}

\section{Notation and reduction techniques}\label{sec:reduction}
In this section, we gather notation and reduction techniques to facilitate the process of lifting pre-Drinfeld orbifold algebra maps $\kappa^L$ to Drinfeld orbifold algebra maps $\kappa=\kappa^L+\kappa^C$.
We first introduce operators on cochains to make it easier to refer to the properties of 
a Drinfeld orbifold algebra map (Definition~\ref{def:fourconditions})
for a group $G$ acting linearly on a vector space $V\cong\CC^n$.
Along the way, we indicate how our notation relates to the cohomological 
interpretation (see~\cite[Theorem~7.2]{SWorbifold}) of each property.
We then record symmetries that reduce the computations involved in clearing
 obstructions.  We use these symmetries heavily in Section~\ref{sec:computations}.

\subsection*{A variation on the coboundary} First, we define a map $\psi$ to compactly describe the left-hand side of Property (\ref{def:fourconditions-ii}) and the right-hand side of Property (\ref{def:fourconditions-iii}) of a Drinfeld orbifold algebra map.
For a linear or constant $2$-cochain $\alpha$,
let $\psi(\alpha)=\sum_{g\in G}\psi_gg$ be the $3$-cochain with components $\psi_g:\bigwedge^3V\rightarrow S(V)$ given by
$$
\psi_g(v_1,v_2,v_3)=\alpha_g(v_1,v_2)(\actby{g}v_3-v_3)
+\alpha_g(v_2,v_3)(\actby{g}v_1-v_1)
+\alpha_g(v_3,v_1)(\actby{g}v_2-v_2).
$$
The map $\psi$ is the negation of the coboundary operator on cochains arising from the Koszul resolution.
In particular, $\psi(\kappa^L)=-d_3^*\kappa^L$ and $\psi(\kappa^C)=-d_3^*\kappa^C$, where
$d_3^*$ is the coboundary operator that takes two-cochains to three-cochains (see the proof of~\cite[Lemma~7.1]{SWorbifold}). 

\subsection*{A variation on the cochain bracket}
Next, we define a map $\phi$ to compactly describe
the left-hand sides of Properties (\ref{def:fourconditions-iii}) and (\ref{def:fourconditions-iv}) of a Drinfeld orbifold algebra map. 
For $\alpha$ a linear or constant $2$-cochain and
$\beta$ a linear $2$-cochain,
let $\phi(\alpha,\beta)=\sum_{g\in G}\phi_g g$ be the $3$-cochain
with components $\phi_g=\sum_{xy=g}\phi_{x,y}$, where $\phi_{x,y}:\bigwedge^3 V\to V\oplus\CC$ is given by
$$
\phi_{x,y}(v_1,v_2,v_3)=\alpha_x(v_1+\actby{y}v_1,\beta_y(v_2,v_3))
+\alpha_x(v_2+\actby{y}v_2,\beta_y(v_3,v_1))
+\alpha_x(v_3+\actby{y}v_3,\beta_y(v_1,v_2)).
$$
Thus $\phi(\alpha,\beta)$ is $G$-graded, with components $\phi_g$, and also  $(G\times G)$-graded, with components $\phi_{x,y}$.
The map $\phi(\alpha,\beta)$ is closely
related to the cochain bracket $[\alpha,\beta]$ in~\cite[Definition~5.6, Corollary~6.7]{SWorbifold}.
Indeed, $\phi(\kappa^L,\kappa^L)=-\frac{1}{2}[\kappa^L,\kappa^L]$ and $\phi(\kappa^C,\kappa^L)=-[\kappa^C,\kappa^L]$, as explained in~\cite[proof of Lemma~7.1]{SWorbifold}.

To the extent possible, and especially in Section~\ref{sec:computations}, we make remarks or calculations that apply to both
$\phi(\kappa^L,\kappa^L)$ and $\phi(\kappa^C,\kappa^L)$, and in these instances,
 we write $\phi^{*}_g$ and
$\phi^{*}_{x,y}$ for the corresponding components of $\phi(\kappa^{*},\kappa^L)$ where $*$ denotes either $L$ or $C$.

\subsection*{Drinfeld orbifold algebra maps (condensed definition)}
Equipped with the definitions of $\psi$ and $\phi$, the properties of a Drinfeld orbifold map $\kappa=\kappa^L+\kappa^C$
(Definition~\ref{def:fourconditions}) may be expressed succinctly:
\begin{itemize}
	\setlength\itemsep{.5em}
	\item[(\ref{def:fourconditions-0})] $\im\kappa^L_g\subseteq V^g$ for each $g$ in $G$,
	\item[(\ref{def:fourconditions-i})] the map $\kappa$ is $G$-invariant,
	\item[(\ref{def:fourconditions-ii})] $\psi(\kappa^L)=0$,
	\item[(\ref{def:fourconditions-iii})] $\phi(\kappa^L,\kappa^L)=2\psi(\kappa^C)$,
	\item[(\ref{def:fourconditions-iv})] $\phi(\kappa^C,\kappa^L)=0$. 
\end{itemize}

\subsection*{Invariance relations} 
Recall that a cochain $\alpha=\sum_{g\in G}\alpha_gg$ with components $\alpha_g:\bigwedge^k V\to S(V)$ is $G$-invariant if and only if $\actby{h}\alpha_g=\alpha_{hgh^{-1}}$ for all $g,h\in G$.
Equivalently,
$$\actby{h}(\alpha_g(v_1,\ldots,v_k))=\alpha_{hgh^{-1}}(\actby{h}v_1,\ldots,\actby{h}v_k)$$
for all $g,h\in G$ and $v_1,\ldots,v_k\in V$. Thus a $G$-invariant cochain is
determined by its components for a set of conjugacy class representatives. 

In the following lemma, one can let $\alpha=\kappa^L$ or $\alpha=\kappa^C$ and let $\beta=\kappa^L$ to see that if $\kappa^L$ and $\kappa^C$ are $G$-invariant,
then $\phi(\kappa^{*},\kappa^L)$ and $\psi(\kappa^{*})$ are also $G$-invariant. This is helpful because, for instance, if $\phi_g=2\psi_g$ for some $g\in G$, then acting by $h\in G$ on both sides shows $\phi_{hgh^{-1}}=2\psi_{hgh^{-1}}$ also. Thus if $\phi_g=2\psi_g$ for all $g$ in a set of conjugacy class representatives, then $\phi(\kappa^L,\kappa^L)=2\psi(\kappa^C)$. Similar
reasoning applies to Properties (\ref{def:fourconditions-ii}) and
(\ref{def:fourconditions-iv}) of a Drinfeld orbifold algebra map.

\begin{lemma}
	\label{le:orbitphixyphig}
	Let $G$ be a finite group acting linearly on $V\cong\CC^n$.  If $\alpha$ and
	$\beta$ are $G$-invariant $2$-cochains with $\beta$ linear and $\alpha$ linear or constant,
	then $\phi(\alpha,\beta)$ and $\psi(\alpha)$ are $G$-invariant.  Specifically, at
	the component level, we have for all $g,h\in G$ and $v_1,v_2,v_3\in V$
	$$
	\actby{h}(\phi_{x,y}(v_1,v_2,v_3))=\phi_{hxh^{-1},hyh^{-1}}(\actby{h}v_1,\actby{h}v_2,\actby{h}v_3),
	$$
	
	$$
	\actby{h}(\phi_g(v_1,v_2,v_3))=\phi_{hgh^{-1}}(\actby{h}v_1,\actby{h}v_2,\actby{h}v_3),
	$$
	and
	$$
	\actby{h}(\psi_g(v_1,v_2,v_3))=\psi_{hgh^{-1}}(\actby{h}v_1,\actby{h}v_2,\actby{h}v_3).
	$$
\end{lemma}

\begin{proof}
	To see that
	$$
	\actby{h}(\phi_{x,y}(v_1,v_2,v_3))=\phi_{hxh^{-1},hyh^{-1}}(\actby{h}v_1,\actby{h}v_2,\actby{h}v_3)
	$$
	for all $x,y,h\in G$ and $v_1,v_2,v_3\in V$,
	note that, using invariance of $\alpha$ and $\beta$,
	\begin{eqnarray*}
		\actby{h}(\alpha_x(v_i+\actby{y}v_i,\beta_y(v_j,v_k))) 
		& = & \alpha_{hxh^{-1}}(\actby{h}v_i+\actby{hy}v_i,\beta_{hyh^{-1}}(\actby{h}v_j,\actby{h}v_k)) \\
		& = &\alpha_{hxh^{-1}}(\actby{h}v_i+\actby{hyh^{-1}}(\actby{h}v_i),\beta_{hyh^{-1}}(\actby{h}v_j,\actby{h}v_k)). 
	\end{eqnarray*}
	Then also
	$$
	\actby{h}\phi_g=\sum_{xy=g}\actby{h}\phi_{x,y}=\sum_{xy=g}\phi_{hxh^{-1},hyh^{-1}}=\phi_{hgh^{-1}},
	$$
    where the last equality holds since the correspondence $(x,y)\leftrightarrow(hxh^{-1},hyh^{-1})$ is a
    bijection between the set of factor pairs of $g$ and
    the set of factor pairs of $hgh^{-1}$.
	
	To see that 
	$\actby{h}(\psi_g(v_1,v_2,v_3))=\psi_{hgh^{-1}}(\actby{h}v_1,\actby{h}v_2,\actby{h}v_3)$,
	again use that 
	$\actby{h}(\alpha_g(v_i,v_j))=\alpha_{hgh^{-1}}(\actby{h}v_i,\actby{h}v_j)$
	and $\actby{hg}v_k=\actby{hgh^{-1}}(\actby{h}v_k)$.
\end{proof}

\subsection*{Orbits of factorizations}
The next observations involve the action of $G$ on $G\times G$
by diagonal (componentwise) conjugation and provide a method for 
narrowing the number of terms and basis triples we must consider
in evaluating $\phi(\kappa^{*},\kappa^L)$ in Section~\ref{sec:computations}.

If expressions $\phi_{x,y}(u,v,w)$ are organized in an array with
rows indexed by factorizations $xy$ of $g$ and columns indexed by basis triples $\{u,v,w\}$, then $\phi_g(u,v,w)$ corresponds to a column sum. Our goal is to use invariance relations to show how to use column sums in a carefully chosen subarray to determine the column sums for the full array.

We first consider the effect of acting on a column sum for a subarray with rows indexed by
factorizations in the same orbit under a subgroup.

\begin{lemma}
	\label{zactionONorbitsum}
	Let $G$ be a finite group acting linearly on $V\cong\CC^n$, and let
	$\alpha$ and $\beta$ be $G$-invariant $2$-cochains with $\beta$ linear
	and $\alpha$ linear or constant. Recall that $\phi(\alpha,\beta)$
	has components $\phi_g=\sum_{xy=g}\phi_{x,y}$.
	Fix $g$ in $G$, and let $H$ be a subgroup of the centralizer $Z(g)$.
	If $g=xy$, then for all $z$ in $Z(g)$ and $u,v,w$ in $V$,
	\begin{eqnarray*}
		\actby{z}\left(\sum_{(x',y')\in\actby{H}(x,y)}\phi_{x',y'}(u,v,w)\right)
		& = &\sum_{(x',y')\in\actby{H}(x,y)}\phi_{\actby{z}x',\actby{z}y'}(\actby{z}u,\actby{z}v,\actby{z}w) \\
		& = &\sum_{(x',y')\in\actby{zH}(x,y)}\phi_{x',y'}(\actby{z}u,\actby{z}v,\actby{z}w).
	\end{eqnarray*}
\end{lemma}

\begin{proof}
	The first equality is an application of Lemma~\ref{le:orbitphixyphig}, and the second equality holds because 
	the elements of $\actby{H}(x,y)$ and $\actby{zH}(x,y)=\{\actby{z}(x',y')\;|\;(x',y')\in\actby{H}(x,y)\}$ are in bijection via the map
	$(x',y')\mapsto\actby{z}(x',y')$.
\end{proof}

We use the following characterization of when ``coset orbits''
of factorizations coincide to ensure there is no double-counting in the proof of Lemma~\ref{le:orbitreduction}.

\begin{lemma}\label{le:orbitpartition}
	Let $G$ be a finite group, and let $g$ be an element of $G$ with factorization $g=xy$.  
	Let $K=Z(x)\cap Z(y)$, the stabilizer of $(x,y)$ under componentwise
	conjugation, and let $H$ be a subgroup of $Z(g)$ normalized by $K$.
	Then the orbits $\actby{z_1H}(x,y)$ and $\actby{z_2H}(x,y)$ are disjoint or equal, with
	equality if and only if $z_1HK=z_2HK$.
\end{lemma}
\begin{proof}
	Let $z_1,z_2\in Z(g)$ and 
	suppose the orbits $\actby{z_1H}(x,y)$ and $\actby{z_2H}(x,y)$ intersect nontrivially so that
	$\actby{z_1h_1}(x,y)=\actby{z_2h_2}(x,y)$ for some $h_1,h_2$ in $H$.
	Then $h_1^{-1}z_1^{-1}z_2h_2\in K$,
	so $z_1^{-1}z_2\in HKH=HK$ (since $K$ normalizes $H$), and hence $z_1HK=z_2HK$.
	Then
	$$
	\actby{z_1H}(x,y)=\actby{z_1HK}(x,y)=\actby{z_2HK}(x,y)=\actby{z_2H}(x,y).
	$$
\end{proof}

The next lemma, stated in the specific case of the symmetric group, provides
a method for using subgroups to reduce the number of expressions $\phi_{x,y}(e_i,e_j,e_k)$ that must be evaluated when verifying $\phi(\alpha,\beta)=0$.  Choosing the subgroup is
a balancing act\textemdash using a small subgroup decreases the number of factorizations to consider but typically increases the number of basis triples, while using a large subgroup increases the number of factorizations to consider but decreases the number of basis triples.

\begin{lemma}\label{le:orbitreduction}
	Let $S_n$ act on $V\cong\CC^n$ by its natural permutation representation, and let	$\alpha$ and $\beta$ be $S_n$-invariant $2$-cochains with $\beta$ linear
	and $\alpha$ linear or constant. Recall that $\phi(\alpha,\beta)$
	has components $\phi_g=\sum_{xy=g}\phi_{x,y}$.
	Suppose $g$ is in $S_n$ and has factorization $g=xy$.  
	Let $K=Z(x)\cap Z(y)$, the stabilizer of $(x,y)$ under componentwise
	conjugation, and let $H$ be a subgroup of $Z(g)$ normalized by $K$.
	Let $\mathcal{B}$ be the set of all three element subsets of $\{e_1,\ldots,e_n\}$,
	and let $\mathcal{B}_H$ be a set of $H$-orbit representatives of $\mathcal{B}$.
	If
	$$
	\sum_{(x',y')\in\actby{H}(x,y)}\phi_{x',y'}(e_i,e_j,e_k)=0\quad\text{for all $\{e_i,e_j,e_k\}\in\mathcal{B}_H$,}
	$$
	then
	$$
	\sum_{(x',y')\in\actby{Z(g)}(x,y)}\phi_{x',y'}(e_i,e_j,e_k)=0\quad\text{for all $\{e_i,e_j,e_k\}\in\mathcal{B}$.}
	$$
\end{lemma}
\begin{proof} 
	Use Lemma~\ref{zactionONorbitsum} with $z$ ranging over the elements of $H$ to show that
	$$
	\sum_{(x',y')\in\actby{H}(x,y)}\phi_{x',y'}(e_{i},e_{j},e_{k})=0 \quad\text{for all $\{e_i,e_j,e_k\}\in\mathcal{B}$.}
	$$
	Then use Lemma~\ref{zactionONorbitsum} again to show for each $z$ in $Z(g)$, 
	$$
	\sum_{(x',y')\in\actby{zH}(x,y)}\phi_{x',y'}(e_{i},e_{j},e_{k})=0 \quad\text{for all $\{e_i,e_j,e_k\}\in\mathcal{B}$.}
	$$
	Let $[Z(g)/HK]$ be a set of left coset representatives of $HK$, and use
	Lemma~\ref{le:orbitpartition} to conclude,
	\begin{eqnarray*}
		& & \sum_{(x',y')\in\actby{Z(g)}(x,y)}\phi_{x',y'}(e_{i},e_{j},e_{k}) \\
		& = & \sum_{z\in [Z(g)/HK]}\left(\sum_{(x',y')\in\actby{zH}(x,y)}\phi_{x',y'}(e_{i},e_{j},e_{k})\right)=
		0 \quad\text{for all $\{e_i,e_j,e_k\}\in\mathcal{B}$.}
	\end{eqnarray*}
\end{proof}

\begin{remark}
	Though Lemma~\ref{le:orbitreduction} is stated for the symmetric group, 
	a similar idea might be useful in other groups where each centralizer acts by monomial matrices with respect to some basis.
\end{remark}

\section{Lifting to deformations}\label{sec:mainresult}

In Section~\ref{sec:cohomology} we determined all pre-Drinfeld orbifold
algebra maps for the symmetric group acting by its natural permutation representation.
When $n\geq4$, the space of such maps is three-dimensional with
one dimension of maps supported only on the identity and two dimensions of
maps supported only on $3$-cycles.  We now determine for which candidate maps $\kappa^L$ there exists a constant $2$-cochain $\kappa^C$ so that $\kappa=\kappa^L+\kappa^C$ also satisfies Properties (\ref{def:fourconditions-iii}) and (\ref{def:fourconditions-iv}) of a Drinfeld orbifold algebra map.  We consider three cases,
$\kappa^L$ supported only on the identity (Theorem~\ref{thm:LieOrbifoldAlgebraMaps}), supported only on $3$-cycles (Theorem~\ref{thm:DrinfeldOrbifoldAlgebraMaps}), and finally a combination
supported on both the identity and $3$-cycles (Proposition~\ref{prop:L1L3incompatible}).

\subsection*{Lie orbifold algebra maps}

The next theorem shows that for the symmetric group acting by its permutation representation, every pre-Drinfeld orbifold algebra map supported only on the identity lifts uniquely to a Lie orbifold algebra map.
The corresponding Lie orbifold algebras are described in Theorem~\ref{Lieorbifold}.

\begin{theorem}
	\label{thm:LieOrbifoldAlgebraMaps}
	The Lie orbifold algebra maps for the symmetric group $S_n$ ($n\geq3$) acting
	on $V\cong\CC^n$ by the natural permutation representation form a
	one-dimensional vector space generated by the map
	$\kappa:\bigwedge^2V\to V\otimes \CC S_n$ with
	$\kappa(e_i,e_j)=e_i-e_j$ for $1\leq i < j\leq n$.
\end{theorem}

\begin{proof}	
	Let $\kappa^L$ be a pre-Drinfeld orbifold algebra map supported only
	on the identity.  
    By Corollary~\ref{cor:summarypreDOA}, we have $\kappa^L(e_i,e_j)=a(e_i-e_j)$
	for some $a\in\CC$.
	It is straightforward to show that $\phi(\kappa^L,\kappa^L)=0$,
	so now Property~(\ref{def:fourconditions-iii}), $\phi(\kappa^L,\kappa^L)=2\psi(\kappa^C)$, holds if and only if $\psi(\kappa^C)=0$.  By Corollary~\ref{constantcocycles}, the $G$-invariant constant $2$-cochains such that $\psi(\kappa^C)=0$ (i.e., satisfying the mixed Jacobi identity) are supported only on $3$-cycles and have
		$$
		\kappa_{(ijk)}^C(e_i,e_j)=\kappa_{(ijk)}^C(e_j,e_k)=\kappa^C_{(ijk)}(e_k,e_i)=c
		$$
		for some scalar $c$. 
		
	Turning to Property~(\ref{def:fourconditions-iv}), if $c=0$, then $\kappa^C\equiv0$, so $\phi(\kappa^C,\kappa^L)=0$,
	and $\kappa=\kappa^L$ is a Lie orbifold algebra map.
	If $c\neq0$, then $\kappa=\kappa^L+\kappa^C$ is not a Lie orbifold
	algebra map since $\phi(\kappa^C,\kappa^L)\neq0$.
	In particular, the component $\phi_{(123)}$ of $\phi(\kappa^C,\kappa^L)$ is nonzero on the basis triple 
	$e_1, e_2, e_3$:
	\[ 
	2[\kappa_{(123)}^C(e_1,e_2-e_3)+\kappa_{(123)}^C(e_2,e_3-e_1)+\kappa_{(123)}^C(e_3,e_1-e_2)]=12c\neq 0. 
	\]
\end{proof}

In general, a lifting need not be unique.
See~\cite[Example 4.3]{SWorbifold} 
 for an example of a cyclic group having a Lie orbifold algebra map
 with $\kappa^L$ and $\kappa^C$ both nonzero.

\subsection*{Other Drinfeld orbifold algebra maps} 

Next we describe all possible Drinfeld orbifold algebra maps supported only 
off of the identity.  We outline the proof here but relegate the
details of clearing the obstructions to Section~\ref{sec:computations}.
The corresponding Drinfeld orbifold algebras are described in
Theorem~\ref{mainexamples}.

\begin{theorem}\label{thm:DrinfeldOrbifoldAlgebraMaps}
	For $S_n$ ($n\geq3$) acting on $V\cong\CC^n$ by its natural permutation representation, the Drinfeld orbifold algebra maps supported only off the identity are precisely the maps of the form $\kappa=\kappa^L_{\tri}+\kappa^C_{\penta}+\kappa^C_{\tri}$,
	with $\kappa_{\tri}$ as in Definition~\ref{def:kappatri} and
	$\kappa^C_{\penta}$ as in Definition~\ref{def:kappa5cyc}.
\end{theorem}
\begin{proof}
	Suppose $\kappa^L$ is a pre-Drinfeld orbifold algebra map supported only
	off of the identity.  By Corollary~\ref{cor:summarypreDOA},
	we must have $\kappa^L=\kappa^L_{\tri}$ for some parameters $a,b\in\CC$ as in Definition~\ref{def:kappatri}. Now, the goal is to find all $G$-invariant maps $\kappa^C$ such that Properties (\ref{def:fourconditions-iii}) and (\ref{def:fourconditions-iv})
	of a Drinfeld orbifold algebra map also hold.
	
	First, we find a particular lifting.
\begin{itemize}
	\item {\it First obstruction.} In Propositions~\ref{prop:zerocases} and~\ref{prop:5-cycle}, we evaluate $\phi(\kappa^L_{\tri},\kappa^L_{\tri})$. The results suggest Definition~\ref{def:kappa5cyc} of an $S_n$-invariant map $\kappa^C_{\penta}$ so that
	$$\phi(\kappa^L_{\tri},\kappa^L_{\tri})=2\psi(\kappa^C_{\penta}),$$
	as verified in Proposition~\ref{prop:phiL3L3=2psiC5}.
	\item {\it Second obstruction.} In Proposition~\ref{prop:phiC5L3=0}, we show $\phi(\kappa^C_{\penta},\kappa^L_{\tri})=0$.
\end{itemize}
	Thus $\kappa=\kappa^L_{\tri}+\kappa^C_{\penta}$ is
	a Drinfeld orbifold algebra map.

Next, we see how the particular lifting can be modified to produce all other possible liftings. Let $\kappa^C$ be any $G$-invariant constant $2$-cochain.
\begin{itemize}
  \item {\it First obstruction. } Given that $\phi(\kappa^L_{\tri},\kappa^L_{\tri})=2\psi(\kappa^C_{\penta})$,
we have that $\phi(\kappa^L_{\tri},\kappa^L_{\tri})=2\psi(\kappa^C)$
if and only if $\psi(\kappa^C-\kappa^C_{\penta})=0$.
  By Corollary~\ref{constantcocycles}, $\psi(\kappa^C-\kappa^C_{\penta})=0$
  if and only if $\kappa^C-\kappa^C_{\penta}=\kappa^C_{\tri}$, with $\kappa^C_{\tri}$
  as in Definition~\ref{def:kappatri} for some parameter $c\in\CC$.
  \item {\it Second obstruction. }
  In Propositions~\ref{prop:zerocases} and~\ref{prop:5-cycle}, we show $\phi(\kappa^C_{\tri},\kappa^L_{\tri})=0$, and in Proposition~\ref{prop:phiC5L3=0}, we show $\phi(\kappa^C_{\penta},\kappa^L_{\tri})=0$,
  so
  $$\phi(\kappa^C_{\tri}+\kappa^C_{\penta},\kappa^L_{\tri})
  =\phi(\kappa^C_{\tri},\kappa^L_{\tri})
  +\phi(\kappa^C_{\penta},\kappa^L_{\tri})=0.$$
 \end{itemize}
  Thus the liftings of $\kappa^L_{\tri}$ to a Drinfeld orbifold algebra map are
  the maps of the form $\kappa=\kappa^L_{\tri}+\kappa^C_{\tri}+\kappa^C_{\penta}$.
\end{proof}

Finally, we show for $S_n$ that there are no Drinfeld orbifold algebra maps that
are supported both on and off the identity.

\begin{proposition}\label{prop:L1L3incompatible}
	Let $\kappa^L$ be a pre-Drinfeld orbifold algebra map for the natural permutation representation of the symmetric group
	$S_n$ ($n\geq3)$.
	If $\kappa^L$ is supported both on the identity
	and off the identity, then $\kappa^L$ does not lift to a Drinfeld orbifold
	algebra map.
\end{proposition}

\begin{proof}
	Let $\kappa^L\not\equiv0$ be a pre-Drinfeld orbifold algebra map.  By Corollary~\ref{cor:summarypreDOA}, we have
	$\kappa^L=\kappa^L_1+\kappa^L_{\tri}$, where $\kappa^L_{\tri}$ is
	as in Definition~\ref{def:kappatri} with parameters $a,b\in\CC$, and
	$\kappa^L_1(e_i,e_j)=a_1(e_i-e_j)$ for some parameter $a_1\in\CC$.  
	Suppose $\kappa^C$ is an $S_n$-invariant constant $2$-cochain.
	We show that if $\kappa^C$ clears the first obstruction, i.e., $\phi(\kappa^L,\kappa^L)=2\psi(\kappa^C)$, then $a=b=0$ or $a_1=0$, so
	$\kappa^L$ is supported only on the identity or only on the $3$-cycles.
	
	 We let $g=(123)$ and compare the $g$-components of $2\psi(\kappa^C)$ and $\phi(\kappa^L,\kappa^L)$.
	First, note that since $\kappa^C$ is $S_n$-invariant, $$\kappa_{(123)}^C(e_1,e_2)=\kappa^C_{(123)}(e_2,e_3)=\kappa^C_{(123)}(e_3,e_1),$$
	and so 
	$$\psi_{(123)}(e_1,e_2,e_3)=
	\kappa^C_{(123)}(e_1,e_2)(e_1-e_3)
	+\kappa^C_{(123)}(e_2,e_3)(e_2-e_1)
	+\kappa^C_{(123)}(e_3,e_1)(e_3-e_2)=0.$$
	On the other hand,
	$$
	\phi(\kappa^L,\kappa^L)=\phi(\kappa^L_1,\kappa^L_1)+\phi(\kappa^L_{\tri},\kappa^L_1)+\phi(\kappa^L_{1},\kappa^L_{\tri})+\phi(\kappa^L_{\tri},\kappa^L_{\tri}).
	$$
	The $(123)$-component of $\phi(\kappa^L_1,\kappa^L_1)$ is certainly zero,
	and by Proposition~\ref{prop:zerocases}, the $(123)$-component of
	$\phi(\kappa^L_{\tri},\kappa^L_{\tri})$ is also zero.
	Only the cross terms remain, so
	$$
	\phi_{(123)}=\phi_{(123),1}+\phi_{1,(123)}.
	$$
	Recall
	\[\phi_{x,y}(e_i,e_j,e_k)=\phiLxy{x}{y}{e_i}{e_j}{e_k}.\]
	The cross terms of $\phi_{(123)}$ evaluated on the basis triple $e_1,e_2,e_3$ are
	\begin{align*}
	\phi_{(123),1}(e_1,e_2,e_3)
	&=2a_1(\kappa^L_{(123)}[e_1,e_2-e_3]
	+\kappa^L_{(123)}[e_2,e_3-e_1]
	+\kappa^L_{(123)}[e_3,e_1-e_2])\\
	&=4a_1(\kappa^L_{(123)}[e_1,e_2]
	+\kappa^L_{(123)}[e_2,e_3]
	+\kappa^L_{(123)}[e_3,e_1])\\
	&=12a_1[a(e_1+e_2+e_3)+b(e_4+\cdots+e_n)]
	\end{align*}
	and
	\begin{align*}
	\phi_{1,(123)}(e_1,e_2,e_3)
	&=2\kappa^L_{1}[e_1+e_2+e_3,\kappa^L_{(123)}(e_1,e_2)]\\
	&=2a_1b[(n-3)(e_1+e_2+e_3)-3(e_4+\cdots+e_n)],
	\end{align*}
	which sum to
	\begin{align*}
	\phi_{(123)}(e_1,e_2,e_3)
	&=2a_1[(6a+(n-3)b)(e_1+e_2+e_3)+3b(e_4+\cdots+e_n)].
	\end{align*}
	Thus, if $n\geq4$, then $\phi_{(123)}(e_1,e_2,e_3)=2\psi_{(123)}(e_1,e_2,e_3)$ if
	and only if $a_1=0$ or $b=0=a$. 
	If $n=3$, then $\kappa^L_{\tri}$ has only one parameter, $a$,
	and $\phi_{(123)}(e_1,e_2,e_3)=12a_1a(e_1+e_2+e_3)$, which is zero if and
	only if $a_1=0$ or $a=0$.	
\end{proof}


\section{Clearing the obstructions}\label{sec:computations}

In Section~\ref{sec:cohomology}, we defined a pre-Drinfeld orbifold algebra map $\kappa^L_{\tri}$.  This section is devoted to lifting $\kappa^L_{\tri}$ to a
Drinfeld orbifold algebra map and provides the details of the proof of Theorem~\ref{thm:DrinfeldOrbifoldAlgebraMaps} outlined in Section~\ref{sec:mainresult}.
In view of Definition~\ref{def:fourconditions}, we first evaluate $\phi(\kappa^L_{\tri},\kappa^L_{\tri})$ and ``clear the first obstruction'' by defining a $G$-invariant map $\kappa^C_{\penta}$ such that $\phi(\kappa^L_{\tri},\kappa^L_{\tri})=2\psi(\kappa^C_{\penta})$.  Existence of
such a map is predicted by \cite[Theorem~9.2]{SWbrackets12}.
We then ``clear the second obstruction'' by showing $\phi(\kappa^C_{\penta},\kappa^L_{\tri})=0$. These computations show that $\kappa=\kappa^L_{\tri}+\kappa^C_{\penta}$ is a Drinfeld orbifold algebra map.

\subsection*{Clearing the First Obstruction}

We begin by recording simplifications of $\phi^{*}_{x,y}$, a summand of the component $\phi_g^{*}$ of $\phi(\kappa^{*}_{\tri},\kappa^L_{\tri})$, where $*$ stands for $L$ or $C$.  Simplification of $\phi^{*}_{x,y}(e_i,e_j,e_k)$ depends on
the location of the basis vectors relative to the fixed spaces $V^x$ and $V^y$,
so we use the following indicator function.
For $y \in S_n$ and $v \in V$, let
\[\delta_y(v)=\begin{cases}
1 & \text{if $v \in V^y$} \\
0 & \text{otherwise}.
\end{cases}\]

\begin{lemma}
	\label{lemma:simplification}
	Let $\kappa^{*}_{\tri}$ with $*=L$ or $*=C$ be as in Definition~\ref{def:kappatri} and let $\phi^{*}_{x,y}$ denote a term of the component $\phi^{*}_g$ of $\phi(\kappa^{*}_{\tri},\kappa^L_{\tri})$.
	Let $x$ and $y$ be $3$-cycles such that $xy=g$, and let $1\leq i,j,k\leq n$.
	\begin{enumerate}[label={(\arabic*)},ref={\thelemma~(\arabic*)}]
		\setlength\itemsep{.5em}
		\item If $e_i,e_j\in V^y$, then 
		$\phi^{*}_{x,y}(e_i,e_j,e_k)=0$. \label{part1}
		\item If $e_i \in V^y\cap V^x$, then $\phi^{*}_{x,y}(e_i,e_j,e_k)=0$. \label{part2}
		\item If $e_i\in V^y\backslash V^x$ and $e_j\not\in V^y$, then  
		\[\phi^{*}_{x,y}(e_i,e_j,\actby{y}e_j)= 2(b-a)[\delta_y(\actby{x}e_i)-\delta_y(\actbylessspace{x^{-1}}e_i)]\kappa_x^{*}[e_i,\actby{x}e_i].\] \label{part3}
		\item If $e_i \notin V^y$, then $\phi^{*}_{x,y}(e_i,\actby{y}e_i,\actbylessspace{y^2}e_i)=0$. \label{part4}
	\end{enumerate}
	Note that $\phi^{*}_{x,y}$ can be evaluated on any basis triple by using the
	alternating property along with these cases.
\end{lemma}

\begin{proof}
	Consider
	\[\phi^{*}_{x,y}(e_i,e_j,e_k)=\phibulletxy{x}{y}{e_i}{e_j}{e_k}.\]
	Recall that if $z$ is a $3$-cycle then $V^z\subseteq \ker\kappa_z^*$.
	\begin{enumerate}
		\item If $e_i,e_j\in V^y$, then 
		$\phi^{*}_{x,y}(e_i,e_j,e_k)=0$ since $V^y\subseteq\ker\kappa_y^L$.
		\item If $e_i\in V^y\cap V^x$, then the first term of $\phi^{*}_{x,y}$ vanishes
		because $e_i\in V^x\subseteq\ker\kappa_x^*$, and the second two terms
		of $\phi^{*}_{x,y}$ vanish because $e_i\in V^y\subseteq\ker\kappa^L_y$.
		\item If $e_i\in V^y\backslash V^x$ and $e_j\not\in V^y$, then $\phi^{*}_{x,y}(e_i,e_j,\actby{y}e_j) = 2\kappa_x^{*}[e_i,\kappa^L_y(e_j,\actby{y}e_j)]$.
		Using bilinearity and $V^x\subseteq\ker\kappa_x^*$, the right hand side is a linear combination of expressions $\kappa_x^{*}[e_i,\actby{h}e_i]$ for $h \in \langle x\rangle$.  The appropriate coefficients, $a$ or $b$, can be extracted using the fixed space indicator function.  Also, $\sum_{h\in\langle x\rangle} \actby{h}e_i \in V^x\subseteq\ker\kappa_x^*$.  Thus,
		\begin{align*}
			\phi^{*}_{x,y}(e_i,e_j,\actby{y}e_j) &= 2\kappa_x^{*}[e_i,\kappa^L_y(e_j,\actby{y}e_j)]\\
			&=2a\sum_{h\in\langle x\rangle}(1-\delta_y(\actby{h}e_i))\kappa_x^{*}[e_i,\actby{h}e_i]+2b\sum_{h\in\langle x\rangle}\delta_y(\actby{h}e_i)\kappa_x^{*}[e_i,\actby{h}e_i] \\
			&=2(b-a)\sum_{h\in\langle x\rangle}\delta_y(\actby{h}e_i)\kappa_x^{*}[e_i,\actby{h}e_i]\\
			&=2(b-a)[\delta_y(\actby{x}e_i)-\delta_y(\actbylessspace{x^{-1}}e_i)]\kappa_x^{*}[e_i,\actby{x}e_i].
		\end{align*}
		\item Lastly, if $e_i\notin V^y$, note that
		$\phi^{*}_{x,y}(e_i,\actby{y}e_i,\actbylessspace{y^2}e_i)=2\kappa^{*}_x[e_i+\actby{y}e_i+\actbylessspace{y^2}e_i,\kappa^L_y(e_i,\actby{y}e_i)]$ since $\kappa^L_y(e_i,\actby{y}e_i)=\kappa^L_y(\actby{y}e_i,\actbylessspace{y^2}e_i)=\kappa^L_y(\actbylessspace{y^2}e_i,e_i)$.
		Express $\kappa^L_y(e_i,\actby{y}e_i)$ as a linear combination of the vector $u=e_i+\actby{y}e_i+\actbylessspace{y^2}e_i$ and the vector $u_0=e_1+\cdots+e_n$ (which is in
		the kernel of $\kappa_x^{*}$), to see that
		\[\phi^{*}_{x,y}(e_i,\actby{y}e_i,\actbylessspace{y^2}e_i)=2\kappa_x^{*}(u,(a-b)u+bu_0)=0.\]	
	\end{enumerate}
\end{proof}


As mentioned in the outline of the proof of Theorem~\ref{thm:DrinfeldOrbifoldAlgebraMaps}, the next two propositions are used to evaluate both $\phi(\kappa^L_{\tri},\kappa^L_{\tri})$ and $\phi(\kappa^C_{\tri},\kappa^L_{\tri})$.

\begin{proposition}\label{prop:zerocases}
	Let $\kappa^{*}_{\tri}$ with $*=L$ or $*=C$ be as in Definition~\ref{def:kappatri}.  
	For $g \in S_n$, let $\phi^{*}_g$ be the $g$-component of $\phi(\kappa^{*}_{\tri},\kappa^L_{\tri})$.  
If $g$ is not a $5$-cycle then $\phi^{*}_g \equiv 0$.
\end{proposition}

\begin{proof} 
	Since $\kappa^{*}_{\tri}$ is supported only on $3$-cycles, we first determine the cycle types that arise as a product $xy$ with $x$ and $y$
	both $3$-cycles.
	Since $\actby{\sigma}x\actby{\sigma}y=\actby{\sigma}(xy)$, it suffices to examine representatives
	of orbits of factor pairs $(x,y)$ under the action of $S_n$ by diagonal conjugation.
	Orbit representatives and their products are
	\begin{equation*}
		\begin{alignedat}{2}
			(123)(456)&, &\qquad (123)(324)&=(124), \\
			(123)(345)&=(12345), &\qquad (123)(123)&=(132), \\
			(123)(234)&=(12)(34), &\qquad (123)(321)&=1.
		\end{alignedat}
	\end{equation*}
	If the cycle type of $g$ does not appear in this list, then certainly $\phi^{*}_g\equiv0$.
	We leave the case $g=(12345)$ to Proposition~\ref{prop:5-cycle} and show here that $\phi^{*}_g \equiv 0$ for $g=1$, $g=(123)(456)$, $g=(12)(34)$, and $g=(123)$, and hence also for their conjugates.
	
	Besides narrowing the set of representative elements $g$ to consider, the list of orbit representatives reveals a way to organize the terms $\phi^{*}_{x,y}$ of $\phi^{*}_g$.
	Specifically, if the cycle type of $g$ occurs with multiplicity $m$ in the list,
	then the factor pairs with product $g$ are in $m$ orbits under the diagonal conjugation action of $Z(g)$, and we can use a representative from each 
	orbit to generate all the terms $\phi^{*}_{x,y}$ needed to evaluate $\phi^{*}_g$.

\begin{case}[{\boldmath $g=1$}]
The identity component $\phi^{*}_1$ of $\phi(\kappa^{*}_{\tri},\kappa^L_{\tri})$ reduces to the sum of terms $\phi^{*}_{x,x^{-1}}$, where $x$ ranges over the set of $3$-cycles in $S_n$.
For each $3$-cycle $x$, we have $\im\kappa^L_{x^{-1}}\subseteq V^{x^{-1}}=V^{x}\subseteq\ker\kappa^{*}_{x}$,
so $\kappa^{*}_x[u,\kappa^L_{x^{-1}}(v,w)]=0$ for all $u,v,w\in V$.
It follows that $\phi^{*}_{x,x^{-1}}\equiv0$ for each $3$-cycle $x\in S_n$,
and hence, $\phi^{*}_1\equiv0$.
\end{case}

\begin{case}[{\boldmath $g=(123)(456)$}]
If $g=(123)(456)$, then the component $\phi^{*}_{g}$ of $\phi(\kappa^{*}_{\tri},\kappa^L_{\tri})$ 
reduces to the sum of two terms, $\phi^{*}_{x,y}+\phi^{*}_{y,x}$,
where $x=(123)$ and $y=(456)$.
Note that $\im\kappa^L_y\subseteq V^x\subseteq\ker\kappa^{*}_x$, so
$\kappa^{*}_x[u,\kappa^L_y(v,w)]=0$ for all $u,v,w\in V$, and the
same holds if $x$ and $y$ are exchanged.
It follows that both terms $\phi^{*}_{x,y}$ and $\phi^{*}_{y,x}$ 
are zero on any triple of vectors, and hence $\phi^{*}_g\equiv0$.
\end{case}

\begin{case}[{\boldmath $g=(12)(34)$}]

Note that $Z((12)(34))=\langle (1324), (12)\rangle\times \Sym_{\{5,\dots,n\}}$, 
 and 
 \begin{multline*}
\phi^{*}_g=\sum_{(x,y)\in\actby{Z(g)}((123),(234))}\phi^{*}_{x,y}=
\phi^{*}_{(123),(234)}+\phi^{*}_{(342),(421)}+\phi^{*}_{(214),(143)}+\phi^{*}_{(431),(312)} \\
+\phi^{*}_{(213),(134)}+\phi^{*}_{(432),(321)}+\phi^{*}_{(124),(243)}+\phi^{*}_{(341),(412)}.
\end{multline*}
 Applying Lemma~\ref{le:orbitreduction} with $H=\langle (1324) \rangle$, which is a normal subgroup of $Z(g)$, yields that to show $\phi^{*}_g \equiv 0$, it is sufficient to prove 
\[\sum_{(x,y)\in\actby{H}((123),(234))}\phi^{*}_{x,y}(e_i,e_j,e_k)=[\phi^{*}_{(123),(234)}+\phi^{*}_{(342),(421)}+\phi^{*}_{(214),(143)}+\phi^{*}_{(431),(312)}](e_i,e_j,e_k)=0\]
for $H$-orbit representatives $\{e_i,e_j,e_k\}$ of the basis triples.
 Since $V^x\cap V^y\subseteq \ker\phi_{x,y}$ holds by Lemma~\ref{part2}, we only consider
 $\{e_i,e_j,e_k\}\subseteq\{e_1, e_2, e_3, e_4\}$, and since the three element subsets of $\{e_1,e_2,e_3,e_4\}$ are in the same $H$-orbit, it suffices to show
\[[\phi^{*}_{(123),(234)}+\phi^{*}_{(342),(421)}+\phi^{*}_{(214),(143)}+\phi^{*}_{(431),(312)}](e_1,e_2,e_3)=0.\]
The first three terms are zero by Lemma~\ref{part3} since $\actby{x}e_i, \actbylessspace{x^{-1}}e_i, \notin V^y$ in all three cases and the fourth term is zero by Lemma~\ref{part4}. 
Hence $\phi^{*}_{(12)(34)}\equiv 0$.
\end{case}

\begin{case}[{\boldmath $g=(123)$}] 
If $g=(123)$, then $Z((123))=\langle (123)\rangle\times \Sym_{\{4,\dots,n\}}$ and
$$\phi^{*}_g
=\phi^{*}_{(132),(132)}
+\sum\limits_{r \in [n]\setminus [3]}
 \phi^{*}_{(12r),(r23)}
+\phi^{*}_{(23r),(r31)}
+\phi^{*}_{(31r),(r12)}.$$
 
We first consider the term $\phi^{*}_{(132),(132)}$ since the factorization $(132)(132)$ is in its own $Z(g)$-orbit under diagonal conjugation.
Note that $\im\kappa^L_{(132)}\subseteq V^{(132)}\subseteq\ker\kappa^{*}_{(132)}$,
so $\kappa^{*}_{(132)}(u,\kappa^{L}_{(132)}(v,w))=0$ for all
$u,v,w\in V$, and thus, $\phi^{*}_{(132),(132)}\equiv0$.

Applying Lemma~\ref{le:orbitreduction} with $H=\langle (123) \rangle$, which is a normal subgroup of $Z(g)$, yields that 
in order to show 
$\sum\limits_{r \in [n]\setminus [3]}
 \phi^{*}_{(12r),(r23)}
+\phi^{*}_{(23r),(r31)}
+\phi^{*}_{(31r),(r12)}=0$,
it is sufficient to prove 
\[\sum_{(x,y)\in\actby{H}((124),(423))}\phi^{*}_{x,y}(e_i,e_j,e_k)=[\phi^{*}_{(124),(423)}+\phi^{*}_{(234),(431)}+\phi^{*}_{(314),(412)}](e_i,e_j,e_k)=0\]
for $H$-representatives $\{e_i,e_j,e_k\}$ of the basis triples.
 By Lemma~\ref{part2} $V^x\cap V^y\subseteq \ker\phi^{*}_{x,y}$, and hence we only consider
 $\{e_i,e_j,e_k\}\subseteq\{e_1, e_2, e_3, e_4\}$.
 The three element subsets of $\{e_1,e_2,e_3,e_4\}$ form two $H$-orbits, with representatives $\{e_1, e_2, e_k\}$ for $k=3$ and $k=4$.
Note that
 \[[\phi^{*}_{(124),(423)}+\phi^{*}_{(234),(431)}+\phi^{*}_{(314),(412)}](e_1,e_2,e_3)=0\]
by Lemma~\ref{part3} applied to all three terms, noting that $\actby{x}e_i, \actbylessspace{x^{-1}}e_i, \notin V^y$ in all cases.
Also
 \[[\phi^{*}_{(124),(423)}+\phi^{*}_{(234),(431)}+\phi^{*}_{(314),(412)}](e_1,e_2,e_4)=0\] 
by Lemma~\ref{part3} with $\actby{x}e_i, \actbylessspace{x^{-1}}e_i, \notin V^y$ applied to the first two terms and Lemma~\ref{part4} applied to the third term.
This verifies that $\phi^{*}_{(123)}\equiv 0$ and completes the proof.
\end{case}
\end{proof}

\begin{proposition}\label{prop:5-cycle}
	Let $\kappa_{\tri}=\kappa^L_{\tri}+\kappa^C_{\tri}$ be as in
	Definition~\ref{def:kappatri}, with parameters $a,b,c\in\CC$, and let
	$\phi_g^{*}$ denote the $g$-component
	of $\phi(\kappa_{\tri}^{*},\kappa_{\tri}^L)$, where $*=L$
	or $*=C$ and $g$ is a $5$-cycle. Then $\phi_g^C\equiv0$.  For $\phi_g^L$, 
	if $e_i\in V^g$, then $\phi_g^L(e_i,e_j,e_k)=0$, and for $1\leq i\leq n$,
	\begin{align*}
		\phi_{g}^L(e_i,\actby{g}e_i,\actbylessspace{g^2}e_i)
		&=2(a-b)^2(e_i+\actby{g}e_i-\actbylessspace{g^{2}}e_i-\actbylessspace{g^{3}}e_i) \hbox{ and } \\
		\phi_{g}^L(e_i,\actby{g}e_i,\actbylessspace{g^{3}}e_i)
		&=2(a-b)^2(-2e_i+2\,\actbylessspace{g^2}e_i+\actbylessspace{g^{3}}e_i-\actbylessspace{g^{4}}e_i).
	\end{align*}
\end{proposition}

\begin{proof}
	It suffices to evaluate $\phi_g^{*}$ for the conjugacy class representative
	$g=(12345)$, since the results for any conjugate of $g$ can be obtained by the orbit property described in Lemma~\ref{le:orbitphixyphig}.
	Note that $Z(g)=\langle(12345)\rangle\times \Sym_{\{6,\dots,n\}}$, and
	as seen in the proof of Proposition~\ref{prop:zerocases}, the factorizations of $g$ as a product of $3$-cycles are all in the same $Z(g)$-orbit under diagonal conjugation, 
	so
	\[\phi_g^{*}
	=\phi_{(123),(345)}^{*}
	+\phi_{(234),(451)}^{*}
	+\phi_{(345),(512)}^{*}
	+\phi_{(451),(123)}^{*}
	+\phi_{(512),(234)}^{*}.\]
	
	Note that for each pair of $3$-cycles $x$ and $y$ with $xy=g=(12345)$ we have $V^g \subseteq V^x\cap V^y$, and  $V^x\cap V^y\subseteq\ker\phi_{x,y}^{*}$ by Lemma~\ref{part2}, so if any of the vectors in a basis triple lie in $V^g$ then $\phi_{g}^{*}$ is zero on that triple.
	This leaves for further consideration only the cases where
	$\{e_i,e_j,e_k\}\subseteq \{e_1,e_2,e_3,e_4,e_5\}$.
	
	First, consider $\phi_g^{*}(e_1,e_2,e_3)$.  By Lemma~\ref{part1}, the terms $\phi_{(123),(345)}^{*}(e_1,e_2,e_3)$ and $\phi_{(234),(451)}^{*}(e_1,e_2,e_3)$ are both zero.
	By Lemma~\ref{part4}, the term $\phi_{(451),(123)}^{*}(e_1,e_2,e_3)$ is also zero. Applying Lemma~\ref{part3} to the remaining terms yields
	\begin{align*}
	\phi_g^{*}(e_1,e_2,e_3)
	&=2(a-b)(\kappa^{*}_{(512)}(e_1,e_2)-\kappa^{*}_{(345)}(e_3,e_4)).
	\end{align*}
		Recall that $\kappa^C_{(ijk)}(e_i,e_j)=c$, 
		$\kappa^L_{(ijk)}(e_i,e_j)=(a-b)(e_i+e_j+e_k)+b(e_1+\cdots+e_n)$, and that $e_I=\sum_{i \in I} e_i$ for $I\subseteq \{1,\dots,n\}$.  Then
	\begin{align*}
	\phi_g^{*}(e_1,e_2,e_3)&=
	\begin{cases}
	2(a-b)^2(e_{\{5,1,2\}}-e_{\{3,4,5\}}) & \text{if $*=L$}, \\
	0 & \text{if $*=C$}
	\end{cases}\\
	&=
	\begin{cases}
	2(a-b)^2(e_1+e_2-e_3-e_4) & \text{if $*=L$}, \\
	0 & \text{if $*=C$}.
	\end{cases}
	\end{align*}
	
	Next, consider $\phi_g^{*}(e_1,e_2,e_4)$.  By Lemma~\ref{part1}, the term $\phi_{(123),(345)}^{*}(e_1,e_2,e_4)$ is zero.  Using the alternating property to apply Lemma~\ref{part3} to the remaining terms yields
	\begin{align*}
	\phi^{*}_g(e_1,e_2,e_4)
	&=2(a-b)(\kappa^{*}_{(234)}(e_2,e_3)+\kappa^{*}_{(345)}(e_4,e_5)-\kappa^{*}_{(451)}(e_4,e_5)-\kappa^{*}_{(512)}(e_1,e_2))\\
	&=
	\begin{cases}
		2(a-b)^2(e_{\{2,3,4\}}+e_{\{3,4,5\}}-e_{\{4,5,1\}}-e_{\{5,1,2\}}) & \text{if $*=L$}, \\
		0 & \text{if $*=C$}
	\end{cases}\\
	&=
	\begin{cases}
		2(a-b)^2(-2e_1+2e_3+e_4-e_5) & \text{if $*=L$}, \\
		0 & \text{if $*=C$}.
	\end{cases}
	\end{align*}
	
	Finally, note that for any $e_i\not\in V^g$, the values of $\phi_g^{*}(e_i,\actby{g}e_i,\actbylessspace{g^2}e_i)$
	and $\phi_{g}^{*}(e_i,\actby{g}e_i,\actbylessspace{g^{3}}e_i)$ are obtained
	from the cases $\phi_g^{*}(e_1,e_2,e_3)$ and $\phi_g^{*}(e_1,e_2,e_4)$
	by acting by an appropriate power of $g$ and using the orbit property
	in Lemma~\ref{le:orbitphixyphig}.
\end{proof}

The next definition of a map $\kappa^C_{\penta}$ supported only on $5$-cycles is motivated by the requirement
$\phi(\kappa^L_{\tri},\kappa^L_{\tri})=2\psi(\kappa^C)$.
When $G$ is $S_3$ or $S_4$ there are no $5$-cycles and $\kappa^C_{\penta}$ is the zero map.

\begin{definition}\label{def:kappa5cyc}
	For parameters $a,b\in\CC$, define an $S_n$-invariant map $\kappa^C_{\penta}=\sum_{g\in S_n}\kappa_g^Cg$ with component maps $\kappa^C_{g}:\bigwedge^2 V\to \CC$.
	If $g$ is not a $5$-cycle, let $\kappa^C_g\equiv0$.  If $g$ is a $5$-cycle, define $\kappa^C_g$ by the skew-symmetric matrix
    $$[\kappa_g^C]=(a-b)^2([g]-[g]^{T}-2[g^2]+2[g^2]^{T}),$$
    where $[g]$ denotes the matrix of $g$ with respect to the basis $e_1,\ldots,e_n$,
    and the $(i,j)$-entry of $[\kappa^C_g]$ records $\kappa^C_g(e_i,e_j)$.
\end{definition}      

In practice we use the consequences that $V^g\subseteq\ker\kappa^C_g$, and if $e_i\not\in V^g$, then
        $$ 
        \kappa^C_g(e_i,\actby{g}e_i)=-(a-b)^2\quad\text{and}\quad
        \kappa^C_g(e_i,\actbylessspace{g^2}e_i)=2(a-b)^2.
        $$
    Also, $\kappa^C_{\penta}$ is $G$-invariant, i.e., $\kappa^C_{hgh^{-1}}(\actby{h}e_i,\actby{h}e_j)=\kappa^C_g(e_i,e_j)$ for all $h,g\in S_n$ and all $1\leq i,j\leq n$.

\begin{proposition}\label{prop:phiL3L3=2psiC5}
  Let $\kappa^L_{\tri}$ and $\kappa^C_{\fivecyc}$ 
  be as in Definitions~\ref{def:kappatri} and \ref{def:kappa5cyc}, with common
  parameters $a,b\in\CC$.  Then $\phi(\kappa^L_{\tri},\kappa^L_{\tri})=2\psi(\kappa^C_{\fivecyc})$.  
\end{proposition}

\begin{proof}
	We compare the component $\phi_g$ of $\phi(\kappa^L_{\tri},\kappa^L_{\tri})$
	with the component $2\psi_g$ of $2\psi(\kappa^C_{\fivecyc})$.
	If $g$ is not a $5$-cycle, then $\phi_g\equiv0$ by Proposition~\ref{prop:zerocases}; and $\kappa_{\penta}^C$ is not supported on $g$, so $2\psi_g\equiv0$ also.
	If $g$ is a $5$-cycle, then note that the results of Proposition~\ref{prop:5-cycle} can be written in the form
\begin{align*}
	\phi_{g}(e_i,\actby{g}e_i,\actbylessspace{g^2}e_i)&=-2(a-b)^2((\actby{g}e_i-e_i)+2(\actbylessspace{g^2}e_i-\actby{g}e_i)+(\actbylessspace{g^3}e_i-\actbylessspace{g^2}e_i)),\\
	\phi_{g}(e_i,\actby{g}e_i,\actbylessspace{g^{3}}e_i)&=2(a-b)^2(2(\actby{g}e_i-e_i)+2(\actbylessspace{g^2}e_i-\actby{g}e_i)-(\actbylessspace{g^4}e_i-\actbylessspace{g^3}e_i)),
\end{align*}
while
\begin{align*}
	\psi_{g}(e_i,\actby{g}e_i,\actbylessspace{g^2}e_i)&=
	\kappa_{g}^C(\actby{g}e_i,\actbylessspace{g^2}e_i)(\actby{g}e_i-e_i)+
	\kappa_{g}^C(\actbylessspace{g^2}e_i,e_i)(\actbylessspace{g^2}e_i-\actby{g}e_i)+
	\kappa_{g}^C(e_i,\actby{g}e_i)(\actbylessspace{g^3}e_i-\actbylessspace{g^2}e_i), \\
	\psi_{g}(e_i,\actby{g}e_i,\actbylessspace{g^{3}}e_i)&=
	\kappa_{g}^C(\actby{g}e_i,\actbylessspace{g^3}e_i)(\actby{g}e_i-e_i) +
	\kappa_{g}^C(\actbylessspace{g^3}e_i,e_i)(\actbylessspace{g^2}e_i-\actby{g}e_i) +
	\kappa_{g}^C(e_i,\actby{g}e_i)(\actbylessspace{g^4}e_i-\actbylessspace{g^3}e_i).
\end{align*}

Finally, if $e_i\in V^g$, then $\phi_g(e_i,e_j,e_k)=0$ and
$$
\psi_g(e_i,e_j,e_k)=
\kappa_g^C(e_j,e_k)(\actby{g}e_i-e_i)+
\kappa_g^C(e_k,e_i)(\actby{g}e_j-e_j)+
\kappa_g^C(e_i,e_j)(\actby{g}e_k-e_k)=0,
$$
where the first term vanishes because $\actby{g}e_i-e_i=0$ and
the second two terms vanish because $V^g\subseteq\ker\kappa_g^C$.
Since $\psi$ is alternating, we see that $\psi_g(e_i,e_j,e_k)=0$ whenever $\{e_i,e_j,e_k\}\cap V^g\neq\varnothing$.
\end{proof}

\subsection*{Clearing the Second Obstruction}

The final step in showing $\kappa=\kappa^L_{\tri}+\kappa^C_{\penta}$ is a Drinfeld orbifold algebra map is to verify $\phi(\kappa^C_{\penta},\kappa^L_{\tri})=0$. 

We begin with a lemma that describes simplifications of the summands
$\phi_{x,y}$ of the components $\phi_g$ of $\phi(\kappa^C_{\penta},\kappa^L_{\tri})$.  As in the analogous Lemma~\ref{lemma:simplification}, simplification of $\phi_{x,y}(e_i,e_j,e_k)$
depends on where the vectors in the basis triple lie relative to the fixed spaces $V^x$ and $V^y$.
Recall that for $\sigma\in S_n$ and $v\in V$, $\delta_{\sigma}(v)=1$ 
if $v\in V^{\sigma}$ and $\delta_{\sigma}(v)=0$ otherwise.

\begin{lemma}
	\label{le:simplification53}
    Let $\kappa_{\penta}^C$ and $\kappa^L_{\tri}$ be as in Definitions~\ref{def:kappa5cyc} and~\ref{def:kappatri}, with common
    parameters $a,b\in\CC$.
	Let $\phi_{x,y}$
	denote a summand of the component $\phi_g$ of $\phi(\kappa^C_{\penta},\kappa^L_{\tri})$. Let $x$ be a $5$-cycle and $y$ be a $3$-cycle.
	Let $e_i,e_j,e_k$ be basis vectors.
	\begin{enumerate}[label={(\arabic*)},ref={\thelemma~(\arabic*)}]
		\setlength\itemsep{.5em}
		\item If $e_i,e_j\in V^y$, then $\phi_{x,y}(e_i,e_j,e_k)=0$. \label{53part1}
		\item If $e_i \in V^y\cap V^x$, then $\phi_{x,y}(e_i,e_j,e_k)=0$. \label{53part2}
		\item If $e_i\in V^y\backslash V^x$ and $e_j \not\in V^y$, then 
		\[
		\phi_{x,y}(e_i,e_j,\actby{y}e_j)=2(a-b)^3
		\left[
		\delta_y(\actby{x}e_i)
		-2\delta_y(\actbylessspace{x^2}e_i)
		+2\delta_y(\actbylessspace{x^{-2}}e_i)
		-\delta_y(\actbylessspace{x^{-1}}e_i)\right].
		\] \label{53part3}
		\item If $e_i\not\in V^y$, then $\phi_{x,y}(e_i,\actby{y}e_i,\actbylessspace{y^2}e_i)=0$. \label{53part4}
	\end{enumerate}
	Note that $\phi_{x,y}$ can be evaluated on any basis triple by using the
	alternating property along with these cases.
\end{lemma}

\begin{proof}
	The proofs of parts (1), (2), and (4) are the same as in the proof of Lemma~\ref{lemma:simplification} since $V^y\subseteq \ker\kappa_y^L$, and $V^x\subseteq\ker\kappa_x^C$ is also true for $\kappa^C_{\penta}$.  The proof of part (3) is the same up until the last step, 
	so
		if $e_i\in V^y\backslash V^x$ and $e_j \not\in V^y$, 
		then
		\begin{align*}
		\phi_{x,y}(e_i,e_j,\actby{y}e_j) 
		&=2(b-a)\sum_{h\in\langle x\rangle}\delta_y(\actby{h}e_i)\kappa_x^C[e_i,\actby{h}e_i]\\
		&=2(a-b)^3\left[
		\delta_y(\actby{x}e_i)
		-2\delta_y(\actbylessspace{x^2}e_i)
		+2\delta_y(\actbylessspace{x^{-2}}e_i)
		-\delta_y(\actbylessspace{x^{-1}}e_i)\right].
		\end{align*}
\end{proof}

The proof of the next proposition uses these simplifications to verify that indeed $\phi(\kappa^C_{\penta},\kappa^L_{\tri})=0$, as mentioned in the outline of the proof of Theorem~\ref{mainexamples}.  This clears the second obstruction.

\begin{proposition}\label{prop:phiC5L3=0}
    Let $\kappa_{\penta}^C$ and $\kappa^L_{\tri}$ be as in Definitions~\ref{def:kappa5cyc} and~\ref{def:kappatri}, with common
    parameters $a,b\in\CC$.
    For every $g\in S_n$, the component $\phi_g$ of $\phi(\kappa_{\penta}^C,\kappa^L_{\tri})$
    is identically zero.
\end{proposition}

\begin{proof}
	Since $\kappa^C_{\penta}$ is supported only on $5$-cycles and $\kappa^L_{\tri}$ is supported only on $3$-cycles, we first determine the cycle types that arise as a product $xy$ with $x$ a $5$-cycle and $y$ a $3$-cycle.
	Since $\actby{\sigma}x\actby{\sigma}y=\actby{\sigma}(xy)$, it suffices to examine representatives
	of orbits of factor pairs $(x,y)$ under the action of $S_n$ by diagonal conjugation.
	Orbit representatives and their products are
	\begin{equation*}
		\begin{alignedat}{2}
(12345)(678), &&\qquad (12345)(567)&=(1234567),   \\
(12345)(456)&=(1234)(56), &\qquad (12345)(546)&=(12346), \\
(12345)(356)&=(123)(456), &\qquad (12345)(536)&=(1236)(45), \\
(12345)(345)&=(12354), &\qquad (12345)(543)&=(123), \\
(12345)(245)&=(12534), &\qquad (12345)(542)&=(12)(34).
		\end{alignedat}
	\end{equation*}
	If the cycle type of $g$ does not appear in this list, then certainly $\phi_g\equiv0$.
	We show further that $\phi_g \equiv 0$ for $g=(12345)(678)$, $g=(1234567)$, $g=(1234)(56)$, $g=(123)(456)$, $g=(12345)$, $g=(12)(34)$, and $g=(123)$, and hence also for their conjugates.
	
	Besides narrowing the set of representative elements $g$ to consider, the list of orbit representatives reveals a way to organize the terms $\phi_{x,y}$ of $\phi_g$.
	Specifically, if the cycle type of $g$ occurs with multiplicity $m$ in the list,
	then the factor pairs with product $g$ are in $m$ orbits under the diagonal conjugation action of $Z(g)$, and we can use a representative from each
	orbit to generate all the terms $\phi_{x,y}$ needed to evaluate $\phi_g$.
 
\setcounter{case}{0}
 \begin{case}[{\boldmath $g=(12345)(678)$}]  
 	Note that $g$ has a unique factorization as a product of a $5$-cycle and a $3$-cycle. To show $\phi_g=\phi_{(12345),(678)}\equiv0$,
 we apply Lemma~\ref{le:orbitreduction} with $H=Z(g)=\langle (12345),(678)\rangle\times \Sym_{\{9,\ldots,n\}}$.
 In particular, it suffices to evaluate $\phi_g$ on $Z(g)$-orbit representatives of the basis triples.  Since $V^x\cap V^y\subseteq \ker\phi_{x,y}$, we only consider
 $\{e_i,e_j,e_k\}\subseteq\{e_l\mid 1\leq l\leq 8\}$. If $\{e_i,e_j,e_k\}$ does not contain
 at least two elements from $\{e_6,e_7,e_8\}$, then $\phi_{x,y}(e_i,e_j,e_k)=0$ by Lemma~\ref{53part1}.
 The remaining triples partition into $Z(g)$-orbits
 $$
 \actby{Z(g)}\{e_1,e_6,e_7\}  \quad\text{and}\quad 
 \actby{Z(g)}\{e_6,e_7,e_8\}.
 $$
 Note that
 $\phi_{x,y}(e_1,e_6,e_7)=0$ by Lemma~\ref{53part3} and
 $\phi_{x,y}(e_6,e_7,e_8)=0$ by Lemma~\ref{53part4}. 
 
 \end{case}
 
 \begin{case}[{\boldmath $g=(1234567)$}]  Note that $Z(g)=\langle (1234567)\rangle\times \Sym_{\{8,\ldots,n\}}$.
 The factorizations of $g$ as a product of a $5$-cycle and a $3$-cycle are in a single $Z(g)$-orbit.
 To show $\phi_g \equiv 0$, it suffices to verify 
 $$
 \sum_{(x,y)\in\actby{Z(g)}((12345),(567))}\phi_{x,y}(e_i,e_j,e_k)=0
 $$
 for $Z(g)$-orbit representatives $\{e_i,e_j,e_k\}$ of the basis triples.
 Since $V^x\cap V^y\subseteq \ker\phi_{x,y}$, we only consider
 $\{e_i,e_j,e_k\}\subseteq\{e_l\mid 1\leq l\leq 7\}$.  The remaining triples partition into $Z(g)$-orbits
 $$
 \actbylessspace{Z(g)}\{e_1,e_2,e_3\}, \quad
 \actbylessspace{Z(g)}\{e_1,e_2,e_4\}, \quad
 \actbylessspace{Z(g)}\{e_1,e_2,e_5\}, \quad
 \actbylessspace{Z(g)}\{e_1,e_2,e_6\}, \quad\text{and}\quad 
 \actbylessspace{Z(g)}\{e_1,e_3,e_5\}.
 $$
 (The sum of orbit sizes is $7+7+7+7+7=\binom{7}{3}$.)
 We simplify each remaining term $\phi_{x,y}(e_i,e_j,e_k)$ of the sum
 and record the results in a table with rows indexed by the elements of $\actby{Z(g)}((12345),(567))$.
 Notice $\phi_{(4 5 6 7 1),(1 2 3)}(e_1,e_2,e_3)=0$ by Lemma~\ref{53part4}, all other zero entries result from Lemma~\ref{53part1}, and all remaining entries are found using Lemma~\ref{53part3}.
 Each column sum is zero.
 $$
 {\renewcommand{\arraystretch}{1.5}
 	\begin{array}{cccccc}
 	\hline
 	x,y & \phi_{x,y}(e_1,e_2,e_3) & \phi_{x,y}(e_1,e_2,e_4) & \phi_{x,y}(e_1,e_2,e_5) & \phi_{x,y}(e_1,e_2,e_6) & \phi_{x,y}(e_1,e_3,e_5) \\
 	\hline
 	(1 2 3 4 5),(5 6 7) & \pdftooltip{0}{By Lemma~\ref{53part1}} & \pdftooltip{0}{By Lemma~\ref{53part1}}  & \pdftooltip{0}{By Lemma~\ref{53part1}} & \pdftooltip{0}{By Lemma~\ref{53part1}} & \pdftooltip{0}{By Lemma~\ref{53part1}} \\
 	
 	(2 3 4 5 6),(6 7 1) & \pdftooltip{0}{By Lemma~\ref{53part1}} & \pdftooltip{0}{By Lemma~\ref{53part1}} & \pdftooltip{0}{By Lemma~\ref{53part1}} & \pdftooltip{-2 (a-b)^3}{By Lemma~\ref{53part3}} & \pdftooltip{0}{By Lemma~\ref{53part1}} \\
 	
 	(3 4 5 6 7),(7 1 2) & \pdftooltip{\phantom{-}2 (a-b)^3}{By Lemma~\ref{53part3}} & \pdftooltip{-4 (a-b)^3}{By Lemma~\ref{53part3}} & \pdftooltip{\phantom{-}4 (a-b)^3}{By Lemma~\ref{53part3}} & \pdftooltip{-2 (a-b)^3}{By Lemma~\ref{53part3}} & \pdftooltip{0}{By Lemma~\ref{53part1}} \\
 	
 	(4 5 6 7 1),(1 2 3) & \pdftooltip{\color{green}0}{By Lemma~\ref{53part4}} & \pdftooltip{\phantom{-}2 (a-b)^3}{By Lemma~\ref{53part3}} & \pdftooltip{-4 (a-b)^3}{By Lemma~\ref{53part3}} & \pdftooltip{\phantom{-}4 (a-b)^3}{By Lemma~\ref{53part3}} & \pdftooltip{\phantom{-}4 (a-b)^3}{By Lemma~\ref{53part3}} \\
 	
 	(5 6 7 1 2),(2 3 4) & \pdftooltip{-2 (a-b)^3}{By Lemma~\ref{53part3}} & \pdftooltip{\phantom{-}2 (a-b)^3}{By Lemma~\ref{53part3}} & \pdftooltip{0}{By Lemma~\ref{53part1}} & \pdftooltip{0}{By Lemma~\ref{53part1}} & \pdftooltip{0}{By Lemma~\ref{53part1}} \\
 	
 	(6 7 1 2 3),(3 4 5) & \pdftooltip{0}{By Lemma~\ref{53part1}} & \pdftooltip{0}{By Lemma~\ref{53part1}} & \pdftooltip{0}{By Lemma~\ref{53part1}} & \pdftooltip{0}{By Lemma~\ref{53part1}} & \pdftooltip{-4 (a-b)^3}{By Lemma~\ref{53part3}} \\
 	
 	(7 1 2 3 4),(4 5 6) & \pdftooltip{0}{By Lemma~\ref{53part1}} & \pdftooltip{0}{By Lemma~\ref{53part1}} & \pdftooltip{0}{By Lemma~\ref{53part1}} & \pdftooltip{0}{By Lemma~\ref{53part1}} & \pdftooltip{0}{By Lemma~\ref{53part1}} \\
 	\hline
 	\end{array}}
 $$
 \end{case}
 \begin{case}[{\boldmath $g=(1234)(56)$}]  Note that $Z(g)=\langle (1234),(56)\rangle\times \Sym_{\{7,\ldots,n\}}$.
 The factorizations of $g$ as a product of a $5$-cycle and a $3$-cycle are in two $Z(g)$-orbits
 $$
 \actby{Z(g)}((12345),(564))\quad\text{and}\quad
 \actby{Z(g)}((12356),(634)).
 $$
 To show $\phi_g \equiv 0$, it suffices to verify 
 $$
 \sum_{(x,y)\in\actby{Z(g)}((12345),(564))}\phi_{x,y}(e_i,e_j,e_k)+
 \sum_{(x,y)\in\actby{Z(g)}((12356),(634))}\phi_{x,y}(e_i,e_j,e_k)=0
 $$
 for $Z(g)$-orbit representatives $\{e_i,e_j,e_k\}$ of the basis triples.
 Since $V^x\cap V^y\subseteq \ker\phi_{x,y}$, we only consider
 $\{e_i,e_j,e_k\}\subseteq\{e_l\mid 1\leq l\leq 6\}$.  The remaining triples partition into $Z(g)$-orbits
 $$
 \actby{Z(g)}\{e_1,e_2,e_3\}, \quad
 \actby{Z(g)}\{e_1,e_2,e_5\}, \quad
 \actby{Z(g)}\{e_1,e_3,e_5\}, \quad\text{and}\quad 
 \actby{Z(g)}\{e_1,e_5,e_6\}.
 $$
 (The sum of orbit sizes is $4+8+4+4=\binom{6}{3}$.)
 We simplify each term $\phi_{x,y}(e_i,e_j,e_k)$ of the sum
 and record the results in a table with rows indexed by the elements of 
 $\actby{Z(g)}((12345),(564))\cup\actby{Z(g)}((12356),(634))$.
 Each zero entry is color-coded and tagged by the applicable part of Lemma~\ref{le:simplification53}, all nonzero entries are found using Lemma~\ref{53part3}, and each column sum is zero.
 $$
 {\renewcommand{\arraystretch}{1.5}
 	\begin{array}{ccccc}
 	\hline
 	x,y & \phi_{x,y}(e_1,e_2,e_3) & \phi_{x,y}(e_1,e_2,e_5) & \phi_{x,y}(e_1,e_3,e_5) & \phi_{x,y}(e_1,e_5,e_6) \\
 	\hline
 	(1 2 3 4 5),(4 5 6) & \pdftooltip{0}{By Lemma~\ref{53part1}} & \pdftooltip{0}{By Lemma~\ref{53part1}} & \pdftooltip{0}{By Lemma~\ref{53part1}} & -2 (a-b)^3 \\
 	
 	(2 3 4 1 5),(1 5 6) & \pdftooltip{0}{By Lemma~\ref{53part1}} & \phantom{-}2 (a-b)^3 & \pdftooltip{\color{red}0}{By Lemma~\ref{53part3}} & \pdftooltip{\color{green}0}{By Lemma~\ref{53part4}} \\
 	
 	(3 4 1 2 5),(2 5 6) & \pdftooltip{0}{By Lemma~\ref{53part1}} & \phantom{-}2 (a-b)^3 & \pdftooltip{0}{By Lemma~\ref{53part1}} & \phantom{-}2 (a-b)^3 \\
 	
 	(4 1 2 3 5),(3 5 6) & \pdftooltip{0}{By Lemma~\ref{53part1}} & \pdftooltip{0}{By Lemma~\ref{53part1}} & \pdftooltip{\color{red}0}{By Lemma~\ref{53part3}} & \pdftooltip{\color{red}0}{By Lemma~\ref{53part3}} \\
 	
 	(1 2 3 4 6),(4 6 5) & \pdftooltip{0}{By Lemma~\ref{53part1}} & \pdftooltip{0}{By Lemma~\ref{53part1}} & \pdftooltip{0}{By Lemma~\ref{53part1}} & \phantom{-}2 (a-b)^3 \\
 	
 	(2 3 4 1 6),(1 6 5) & \pdftooltip{0}{By Lemma~\ref{53part1}} & -2 (a-b)^3 & \pdftooltip{\color{red}0}{By Lemma~\ref{53part3}} & \pdftooltip{\color{green}0}{By Lemma~\ref{53part4}} \\
 	
 	(3 4 1 2 6),(2 6 5) & \pdftooltip{0}{By Lemma~\ref{53part1}} & -2 (a-b)^3 & \pdftooltip{0}{By Lemma~\ref{53part1}} & -2 (a-b)^3 \\
 	
 	(4 1 2 3 6),(3 6 5) & \pdftooltip{0}{By Lemma~\ref{53part1}} & \pdftooltip{0}{By Lemma~\ref{53part1}} & \pdftooltip{\color{red}0}{By Lemma~\ref{53part3}} & \pdftooltip{\color{red}0}{By Lemma~\ref{53part3}} \\
 	\hline
 	(1 2 3 5 6),(3 4 6) & \pdftooltip{0}{By Lemma~\ref{53part1}} & \pdftooltip{0}{By Lemma~\ref{53part1}} & \pdftooltip{0}{By Lemma~\ref{53part1}} & \pdftooltip{0}{By Lemma~\ref{53part1}} \\
 	
 	(2 3 4 5 6),(4 1 6) & \pdftooltip{0}{By Lemma~\ref{53part1}} & \pdftooltip{0}{By Lemma~\ref{53part1}} & \pdftooltip{0}{By Lemma~\ref{53part1}} & \pdftooltip{\color{red}0}{By Lemma~\ref{53part3}} \\
 	
 	(3 4 1 5 6),(1 2 6) & \phantom{-}6 (a-b)^3 & \pdftooltip{\color{red}0}{By Lemma~\ref{53part3}} & \pdftooltip{0}{By Lemma~\ref{53part1}} & \pdftooltip{\color{red}0}{By Lemma~\ref{53part3}} \\
 	
 	(4 1 2 5 6),(2 3 6) & -6 (a-b)^3 & \pdftooltip{0}{By Lemma~\ref{53part1}} & \pdftooltip{0}{By Lemma~\ref{53part1}} & \pdftooltip{0}{By Lemma~\ref{53part1}} \\
 	
 	(1 2 3 6 5),(3 4 5) & \pdftooltip{0}{By Lemma~\ref{53part1}} & \pdftooltip{0}{By Lemma~\ref{53part1}} & -6 (a-b)^3 & \pdftooltip{0}{By Lemma~\ref{53part1}} \\
 	
 	(2 3 4 6 5),(4 1 5) & \pdftooltip{0}{By Lemma~\ref{53part1}} & -6 (a-b)^3 & \phantom{-}6 (a-b)^3 & \pdftooltip{\color{red}0}{By Lemma~\ref{53part3}} \\
 	
 	(3 4 1 6 5),(1 2 5) & \phantom{-}6 (a-b)^3 & \pdftooltip{\color{green}0}{By Lemma~\ref{53part4}} & \phantom{-}6 (a-b)^3 & \pdftooltip{\color{red}0}{By Lemma~\ref{53part3}} \\
 	
 	(4 1 2 6 5),(2 3 5) & -6 (a-b)^3 & \phantom{-}6 (a-b)^3 & -6 (a-b)^3 & \pdftooltip{0}{By Lemma~\ref{53part1}} \\
 	\hline
 	\end{array}}
 $$
 \end{case}
 \begin{case}[{\boldmath $g=(123)(456)$}]  Note that $Z(g)=\langle (123),(456),(14)(25)(36)\rangle\times \Sym_{\{7,\ldots,n\}}$, 
 and the subgroup
 $$
 H=\langle (123),(456)\rangle
 $$
 is normal in $Z(g)$.
 The factorizations of $g$ as a product of a $5$-cycle and a $3$-cycle are in a single $Z(g)$-orbit (of size $18$), which partitions into two $H$-orbits (of size $9$)
 $$
 \actby{Z(g)}((12345),(563))=\actby{H}((12345),(563))\cup\actby{H}((45612),(236)).
 $$
 To show $\phi_g \equiv 0$, it suffices to verify 
 $$
 \sum_{(x,y)\in\actby{H}((12345),(563))}\phi_{x,y}(e_i,e_j,e_k)=0
 $$
 for $H$-orbit representatives $\{e_i,e_j,e_k\}$ of the basis triples.
 Since $V^x\cap V^y\subseteq \ker\phi_{x,y}$, we only consider
 $\{e_i,e_j,e_k\}\subseteq\{e_l\mid 1\leq l\leq 6\}$.  
 The remaining triples partition into $H$-orbits
 $$
 \actby{H}\{e_1,e_2,e_4\}, \quad
 \actby{H}\{e_1,e_4,e_5\}, \quad
 \actby{H}\{e_1,e_2,e_3\}, \quad\text{and}\quad 
 \actby{H}\{e_4,e_5,e_6\}.
 $$
 (The sum of orbit sizes is $9+9+1+1=\binom{6}{3}$.)
 We simplify each term $\phi_{x,y}(e_i,e_j,e_k)$ of the sum
 and record the results in a table with rows indexed by the elements of $\actby{H}((12345),(563))$.
 Each zero entry is color-coded and tagged by the applicable part of Lemma~\ref{le:simplification53}, all nonzero entries are found using Lemma~\ref{53part3}, and each column sum is zero.
 $$
 {\renewcommand{\arraystretch}{1.5}
 	\begin{array}{ccccc}
 	\hline
 	x,y & \phi_{x,y}(e_1,e_2,e_3) & \phi_{x,y}(e_4,e_5,e_6) & \phi_{x,y}(e_1,e_2,e_4) & \phi_{x,y}(e_1,e_4,e_5) \\
 	\hline
 	(1 2 3 4 5),(5 6 3) & \pdftooltip{0}{By Lemma~\ref{53part1}} & \pdftooltip{\color{red}0}{By Lemma~\ref{53part3}} & \pdftooltip{0}{By Lemma~\ref{53part1}} & \pdftooltip{0}{By Lemma~\ref{53part1}} \\
 	
 	(1 2 3 5 6),(6 4 3) & \pdftooltip{0}{By Lemma~\ref{53part1}} & \pdftooltip{\color{red}0}{By Lemma~\ref{53part3}} & \pdftooltip{0}{By Lemma~\ref{53part1}} & \pdftooltip{0}{By Lemma~\ref{53part1}} \\
 	
 	(1 2 3 6 4),(4 5 3) & \pdftooltip{0}{By Lemma~\ref{53part1}} & \pdftooltip{\color{red}0}{By Lemma~\ref{53part3}} & \pdftooltip{0}{By Lemma~\ref{53part1}} & \phantom{-}6 (a-b)^3 \\
 	
 	(2 3 1 4 5),(5 6 1) & \pdftooltip{0}{By Lemma~\ref{53part1}} & \pdftooltip{\color{red}0}{By Lemma~\ref{53part3}} & \pdftooltip{0}{By Lemma~\ref{53part1}} & \pdftooltip{\color{red}0}{By Lemma~\ref{53part3}} \\
 	
 	(2 3 1 5 6),(6 4 1) & \pdftooltip{0}{By Lemma~\ref{53part1}} & \pdftooltip{\color{red}0}{By Lemma~\ref{53part3}} & \phantom{-}6 (a-b)^3 & \pdftooltip{\color{red}0}{By Lemma~\ref{53part3}} \\
 	
 	(2 3 1 6 4),(4 5 1) & \pdftooltip{0}{By Lemma~\ref{53part1}} & \pdftooltip{\color{red}0}{By Lemma~\ref{53part3}} & -6 (a-b)^3 & \pdftooltip{\color{green}0}{By Lemma~\ref{53part4}} \\
 	
 	(3 1 2 4 5),(5 6 2) & \pdftooltip{0}{By Lemma~\ref{53part1}} & \pdftooltip{\color{red}0}{By Lemma~\ref{53part3}} & \pdftooltip{0}{By Lemma~\ref{53part1}} & \pdftooltip{0}{By Lemma~\ref{53part1}} \\
 	
 	(3 1 2 5 6),(6 4 2) & \pdftooltip{0}{By Lemma~\ref{53part1}} & \pdftooltip{\color{red}0}{By Lemma~\ref{53part3}} & \phantom{-}6 (a-b)^3 & \pdftooltip{0}{By Lemma~\ref{53part1}} \\
 	
 	(3 1 2 6 4),(4 5 2) & \pdftooltip{0}{By Lemma~\ref{53part1}} & \pdftooltip{\color{red}0}{By Lemma~\ref{53part3}} & -6 (a-b)^3 & -6 (a-b)^3 \\
 	\hline
 	\end{array}}
 $$
 \end{case}
 \begin{case}[{\boldmath $g=(12345)$}]  Note that $Z(g)=H\times \Sym_{\{6,\ldots,n\}}$,
 where $H$ is the normal subgroup
 $$
 H=\langle (12345)\rangle.
 $$
 The factorizations of $g$ as a product of a $5$-cycle and a $3$-cycle are in three $Z(g)$-orbits
 $$
 \actby{Z(g)}((12354),(354)), \quad
 \actby{Z(g)}((12534),(254)),\quad\text{and}\quad
 \actby{Z(g)}((12346),(645)),
 $$
 which partition into $H$-orbits
 \begin{align*}
 \actby{Z(g)}((12354),(354)) &=\actby{H}((12354),(354)), \\
 \actby{Z(g)}((12534),(254)) &=\actby{H}((12534),(254)),\quad\text{and}\\
 \actby{Z(g)}((12346),(645)) &=\bigcup_{r\geq6}\actby{H}((1234r),(r45)).
 \end{align*}
 To show $\phi_g \equiv 0$, it suffices to verify 
 \begin{eqnarray*}
 	& & \sum_{(x,y)\in\actby{H}((12354),(354))}\phi_{x,y}(e_i,e_j,e_k) \\
 	& + & \sum_{(x,y)\in\actby{H}((12534),(254))}\phi_{x,y}(e_i,e_j,e_k) \\
 	& + & \sum_{(x,y)\in\actby{H}((12346),(645))}\phi_{x,y}(e_i,e_j,e_k)=0
 \end{eqnarray*}
 for $H$-orbit representatives $\{e_i,e_j,e_k\}$ of the basis triples.
 Since $V^x\cap V^y\subseteq \ker\phi_{x,y}$, we only consider
 $\{e_i,e_j,e_k\}\subseteq\{e_l\mid 1\leq l\leq 6\}$.  
 The remaining triples partition into $H$-orbits
 $$
 \actby{H}\{e_1,e_2,e_3\}, \quad
 \actby{H}\{e_1,e_2,e_4\}, \quad
 \actby{H}\{e_1,e_2,e_6\}, \quad\text{and}\quad 
 \actby{H}\{e_1,e_3,e_6\}.
 $$
 (The sum of orbit sizes is $5+5+5+5=\binom{6}{3}$.)
 We simplify each term $\phi_{x,y}(e_i,e_j,e_k)$ of the sum
 and record the results in a table with rows indexed by the elements of 
 $\actby{Z(g)}((12354),(354))\cup\actby{Z(g)}((12534),(254))\cup\actby{Z(g)}((12346),(645))$.
 Each zero entry is color-coded and tagged by the applicable part of Lemma~\ref{le:simplification53}, all nonzero entries are found using Lemma~\ref{53part3}, and each column sum is zero.
 $$
 {\renewcommand{\arraystretch}{1.5}
 	\begin{array}{ccccc}
 	\hline
 	x,y & \phi_{x,y}(e_1,e_2,e_3) & \phi_{x,y}(e_1,e_2,e_4) & \phi_{x,y}(e_1,e_2,e_6) & \phi_{x,y}(e_1,e_3,e_6) \\
 	\hline
 	(1 2 3 5 4),(3 5 4) & \pdftooltip{0}{By Lemma~\ref{53part1}} & \pdftooltip{0}{By Lemma~\ref{53part1}} & \pdftooltip{0}{By Lemma~\ref{53part1}} & \pdftooltip{0}{By Lemma~\ref{53part1}} \\
 	
 	(2 3 4 1 5),(4 1 5) & \pdftooltip{0}{By Lemma~\ref{53part1}} & \phantom{-}2 (a-b)^3 & \pdftooltip{0}{By Lemma~\ref{53part1}} & \pdftooltip{0}{By Lemma~\ref{53part1}} \\
 	
 	(3 4 5 2 1),(5 2 1) & -2 (a-b)^3 & \phantom{-}2 (a-b)^3 & \pdftooltip{\color{blue}0}{By Lemma~\ref{53part2}} & \pdftooltip{0}{By Lemma~\ref{53part1}} \\
 	
 	(4 5 1 3 2),(1 3 2) & \pdftooltip{\color{green}0}{By Lemma~\ref{53part4}} & -2 (a-b)^3 & \pdftooltip{\color{blue}0}{By Lemma~\ref{53part2}} & \pdftooltip{\color{blue}0}{By Lemma~\ref{53part2}} \\
 	
 	(5 1 2 4 3),(2 4 3) & \phantom{-}2 (a-b)^3 & -2 (a-b)^3 & \pdftooltip{0}{By Lemma~\ref{53part1}} & \pdftooltip{0}{By Lemma~\ref{53part1}} \\
 	\hline
 	(1 2 5 3 4),(2 5 4) & \pdftooltip{0}{By Lemma~\ref{53part1}} & -4 (a-b)^3 & \pdftooltip{0}{By Lemma~\ref{53part1}} & \pdftooltip{0}{By Lemma~\ref{53part1}} \\
 	
 	(2 3 1 4 5),(3 1 5) & \phantom{-}4 (a-b)^3 & \pdftooltip{0}{By Lemma~\ref{53part1}} & \pdftooltip{0}{By Lemma~\ref{53part1}} & \pdftooltip{\color{blue}0}{By Lemma~\ref{53part2}} \\
 	
 	(3 4 2 5 1),(4 2 1) & -4 (a-b)^3 & \pdftooltip{\color{green}0}{By Lemma~\ref{53part4}} & \pdftooltip{\color{blue}0}{By Lemma~\ref{53part2}} & \pdftooltip{0}{By Lemma~\ref{53part1}} \\
 	
 	(4 5 3 1 2),(5 3 2) & \phantom{-}4 (a-b)^3 & \pdftooltip{0}{By Lemma~\ref{53part1}} & \pdftooltip{0}{By Lemma~\ref{53part1}} & \pdftooltip{0}{By Lemma~\ref{53part1}} \\
 	
 	(5 1 4 2 3),(1 4 3) & -4 (a-b)^3 & \phantom{-}4 (a-b)^3 & \pdftooltip{0}{By Lemma~\ref{53part1}} & \pdftooltip{\color{blue}0}{By Lemma~\ref{53part2}} \\
 	\hline
 	(1 2 3 4 6),(6 4 5) & \pdftooltip{0}{By Lemma~\ref{53part1}} & \pdftooltip{0}{By Lemma~\ref{53part1}} & \pdftooltip{0}{By Lemma~\ref{53part1}} & \pdftooltip{0}{By Lemma~\ref{53part1}} \\
 	
 	(2 3 4 5 6),(6 5 1) & \pdftooltip{0}{By Lemma~\ref{53part1}} & \pdftooltip{0}{By Lemma~\ref{53part1}} & \phantom{-}2 (a-b)^3 & \pdftooltip{\color{red}0}{By Lemma~\ref{53part3}} \\
 	
 	(3 4 5 1 6),(6 1 2) & -2 (a-b)^3 & \pdftooltip{\color{red}0}{By Lemma~\ref{53part3}} & \pdftooltip{\color{green}0}{By Lemma~\ref{53part4}} & -2 (a-b)^3 \\
 	
 	(4 5 1 2 6),(6 2 3) & \phantom{-}2 (a-b)^3 & \pdftooltip{0}{By Lemma~\ref{53part1}} & -2 (a-b)^3 & \phantom{-}2 (a-b)^3 \\
 	
 	(5 1 2 3 6),(6 3 4) & \pdftooltip{0}{By Lemma~\ref{53part1}} & \pdftooltip{0}{By Lemma~\ref{53part1}} & \pdftooltip{0}{By Lemma~\ref{53part1}} & \pdftooltip{\color{red}0}{By Lemma~\ref{53part3}} \\
 	\hline
 	\end{array}}
 $$
 \end{case}
 \begin{case}[{\boldmath $g=(12)(34)$}]  Note that $Z(g)=H\times \Sym_{\{5,\ldots,n\}}$, where $H$ is the
 dihedral group
 $$
 H=\langle (1324),(12)\rangle=\{1,(1324),(12)(34),(4231),(12),(13)(24),(34),(14)(23)\}.
 $$
 The factorizations of $g$ as a product of a $5$-cycle and a $3$-cycle are in a single $Z(g)$-orbit, which we partition into $H$-orbits
 $$
 \actby{Z(g)}((12345),(542))=\bigcup_{r\geq5}\actby{H}((1234r),(r42)).
 $$
 To show $\phi_g \equiv 0$, it suffices to verify 
 $$
 \sum_{(x,y)\in\actby{H}((12345),(542))}\phi_{x,y}(e_i,e_j,e_k)=0
 $$
 for $H$-orbit representatives $\{e_i,e_j,e_k\}$ of the basis triples.
 Since $V^x\cap V^y\subseteq \ker\phi_{x,y}$, we only consider
 $\{e_i,e_j,e_k\}\subseteq\{e_l\mid 1\leq l\leq 5\}$.  
 The remaining triples partition into $H$-orbits
 $$
 \actby{H}\{e_1,e_2,e_3\},\quad
 \actby{H}\{e_1,e_2,e_5\},\quad\text{and}\quad
 \actby{H}\{e_1,e_3,e_5\}.
 $$
 (The sum of the orbit sizes is $4+2+4=\binom{5}{3}$.)
 We simplify each term $\phi_{x,y}(e_i,e_j,e_k)$ of the sum
 and record the results in a table with rows indexed by the elements of $\actby{H}((12345),(542))$.
 Notice that $\phi_{(2 1 4 3 5),(5 3 1)}(e_1,e_3,e_5)=\phi_{(4 3 2 1 5),(5 1 3)}(e_1,e_3,e_5)=0$ follows from Lemma~\ref{53part4}, all other zero entries result from Lemma~\ref{53part1}, and all nonzero entries are found using Lemma~\ref{53part3}.
 Each column sum is zero.
 $$
 {\renewcommand{\arraystretch}{1.5}
 	\begin{array}{cccc}
 	\hline
 	x,y & \phi_{x,y}(e_1,e_2,e_3) & \phi_{x,y}(e_1,e_2,e_5) & \phi_{x,y}(e_1,e_3,e_5) \\
 	\hline
 	(1 2 3 4 5),(5 4 2) & \pdftooltip{0}{By Lemma~\ref{53part1}} & -4 (a-b)^3 & \pdftooltip{0}{By Lemma~\ref{53part1}} \\
 	
 	(1 2 4 3 5),(5 3 2) & \phantom{-}4 (a-b)^3 & -4 (a-b)^3 & \phantom{-}4 (a-b)^3 \\
 	
 	(2 1 3 4 5),(5 4 1) & \pdftooltip{0}{By Lemma~\ref{53part1}} & \phantom{-}4 (a-b)^3 & -4 (a-b)^3 \\
 	
 	(2 1 4 3 5),(5 3 1) & -4 (a-b)^3 & \phantom{-}4 (a-b)^3 & \pdftooltip{\color{green}0}{By Lemma~\ref{53part4}} \\
 	
 	(3 4 1 2 5),(5 2 4) & \pdftooltip{0}{By Lemma~\ref{53part1}} & -4 (a-b)^3 & \pdftooltip{0}{By Lemma~\ref{53part1}} \\
 	
 	(3 4 2 1 5),(5 1 4) & \pdftooltip{0}{By Lemma~\ref{53part1}} & \phantom{-}4 (a-b)^3 & -4 (a-b)^3 \\
 	
 	(4 3 1 2 5),(5 2 3) & \phantom{-}4 (a-b)^3 & -4 (a-b)^3 & \phantom{-}4 (a-b)^3 \\
 	
 	(4 3 2 1 5),(5 1 3) & -4 (a-b)^3 & \phantom{-}4 (a-b)^3 & \pdftooltip{\color{green}0}{By Lemma~\ref{53part4}} \\
 	\hline
 	\end{array}}
 $$
 \end{case}
 \begin{case}[{\boldmath $g=(123)$}]  Note that $Z(g)=\langle (123)\rangle\times \Sym_{\{4,\ldots,n\}}$.
 The factorizations of $g$ as a product of a $5$-cycle and a $3$-cycle are in a single $Z(g)$-orbit
 $$
 \actby{Z(g)}((12345),(543))=\{((123rs),(sr3))\mid \{r,s\}\subseteq\{4,\ldots,n\}\}.
 $$
 Consider the subgroup $H=\langle (123), (45)\rangle$ of $Z(g)$.
 To show $\phi_g \equiv 0$, it suffices to verify 
 $$
 \sum_{(x,y)\in\actby{H}((12345),(543))}\phi_{x,y}(e_i,e_j,e_k)=0
 $$
 for $H$-orbit representatives $\{e_i,e_j,e_k\}$ of the basis triples.
 Since $V^x\cap V^y\subseteq \ker\phi_{x,y}$, we only consider
 $\{e_i,e_j,e_k\}\subseteq\{e_l\mid 1\leq l\leq 5\}$.  
 The remaining triples partition into $H$-orbits
 $$
 \actby{H}\{e_1,e_2,e_3\},\quad
 \actby{H}\{e_1,e_2,e_4\},\quad\text{and}\quad
 \actby{H}\{e_1,e_4,e_5\}.
 $$
 (The sum of the orbit sizes is $1+6+3=\binom{5}{3}$.)
 We simplify each term $\phi_{x,y}(e_i,e_j,e_k)$ of the sum
 and record the results in a table with rows indexed by the elements of $\actby{H}((12345),(543))$.
 Notice that $\phi_{(2 3 1 5 4),(4 5 1)}(e_1,e_4,e_5)=\phi_{(2 3 1 4 5),(5 4 1)}(e_1,e_4,e_5)=0$ follows from Lemma~\ref{53part4}, all other zero entries result from Lemma~\ref{53part1}, and all nonzero entries are found using Lemma~\ref{53part3}.
 Each column sum is zero.
 $$
 {\renewcommand{\arraystretch}{1.5}
 	\begin{array}{cccc}
 	\hline
 	x,y & \phi_{x,y}(e_1,e_2,e_3) & \phi_{x,y}(e_1,e_2,e_4) & \phi_{x,y}(e_1,e_4,e_5) \\
 	\hline
 	(1 2 3 5 4),(4 5 3) & \pdftooltip{0}{By Lemma~\ref{53part1}} & \pdftooltip{0}{By Lemma~\ref{53part1}} & \phantom{-}2 (a-b)^3 \\
 	
 	(1 2 3 4 5),(5 4 3) & \pdftooltip{0}{By Lemma~\ref{53part1}} & \pdftooltip{0}{By Lemma~\ref{53part1}} & -2 (a-b)^3 \\
 	
 	(2 3 1 5 4),(4 5 1) & \pdftooltip{0}{By Lemma~\ref{53part1}} & -2 (a-b)^3 & \pdftooltip{\color{green}0}{By Lemma~\ref{53part4}} \\
 	
 	(2 3 1 4 5),(5 4 1) & \pdftooltip{0}{By Lemma~\ref{53part1}} & \phantom{-}2 (a-b)^3 & \pdftooltip{\color{green}0}{By Lemma~\ref{53part4}} \\
 	
 	(3 1 2 5 4),(4 5 2) & \pdftooltip{0}{By Lemma~\ref{53part1}} & -2 (a-b)^3 & -2 (a-b)^3 \\
 	
 	(3 1 2 4 5),(5 4 2) & \pdftooltip{0}{By Lemma~\ref{53part1}} & \phantom{-}2 (a-b)^3 & \phantom{-}2 (a-b)^3 \\
 	\hline
 	\end{array}}
 $$
 \end{case}
\end{proof}


\bibliographystyle{alpha}
\bibliography{drinfeldorbifoldref}

\end{document}